%% file: 0_Main.tex
\documentclass[abstracton, paper=a4, fontsize=11pt,DIV=12,bibliography=totoc,final]{scrartcl}

\pdfoutput=1

\input{header.tex}

\begin{document}

\title{\LARGE Sedimentation of Particles with Very Small Inertia II:  Derivation, Cauchy Problem and Hydrodynamic Limit of the Vlasov-Stokes Equation}
\author[1]{Richard M. H\"ofer\thanks{richard.hoefer@ur.de}}
\author[2]{Richard Schubert\thanks{schubert@iam.uni-bonn.de}}
\affil[1]{Faculty of Mathematics, University of Regensburg, Germany}
\affil[2]{Institute for Applied Mathematics, University of Bonn, Germany}

\maketitle

\begin{abstract}
    We consider the sedimentation of $N$ spherical particles with identical radii $R$ in a  Stokes flow in $\R^3$. The particles satisfy a no-slip boundary condition and are subject to constant gravity. The dynamics of the particles is modeled by Newton's law but with very small particle inertia as $N$ tends to infinity and $R$ to $0$.
   In a mean-field scaling we show that the evolution of the $N$-particle system is well approximated by the Vlasov-Stokes equation. In contrast to the transport-Stokes equation considered in the first part of this series, \cite{HoferSchubert23},  the Vlasov-Stokes equation takes into account the (small) inertia. Therefore we obtain improved error estimates.

   We also improve previous results on the Cauchy problem for the Vlasov-Stokes equation and on its convergence to the
   transport-Stokes equation in the limit of vanishing inertia.

   The proofs are based on relative energy estimates. 
   In particular, we show new stability estimates for the Vlasov-Stokes equation in the $2$-Wasserstein distance.  By combining a Lagrangian approach with a study of the energy dissipation, we obtain uniform stability estimates for arbitrary small particle inertia.
   We show that a corresponding stability estimate continues to hold for the empirical particle density which formally solves the Vlasov-Stokes equation up to an error. To this end we exploit certain uniform control on the particle configuration thanks to results in the first part \cite{HoferSchubert23}.
\end{abstract}

\tableofcontents

\input{1_Introduction.tex}

\input{2_MainResults.tex}

\input{3_Brinkman}

\input{4_Existence.tex}

\input{5_HydrodynamicLimit}

\input{6_Stability}

\section*{Acknowledgements}

The authors are grateful for the opportunity of an intensive two-week ''Research in pairs`` stay at ''Mathematisches Forschungszentrum Oberwolfach``  which was the starting point of this article. R.H. thanks Barbara Niethammer and the Hausdorff Center for Mathematics for the hospitality during the stays in Bonn.

R.H. has been supported  by the German National Academy of Science Leopoldina, grant LPDS 2020-10.

R.S. has been supported by the Deutsche Forschungsgemeinschaft (DFG, German Research Foundation) through the collaborative research
centre ‘The mathematics of emergent effects’ (CRC 1060, Project-ID 211504053).

\input{Appendix}
 \begin{refcontext}[sorting=nyt]
\printbibliography
 \end{refcontext}
\end{document}

%% file: header.tex
\usepackage[utf8]{inputenc}
\usepackage[T1]{fontenc}
\usepackage{lmodern}
\usepackage[english]{babel}
\usepackage{csquotes}
\usepackage{amsthm, amssymb, amsmath, amsfonts, mathrsfs, dsfont, esint,textcomp}
\usepackage[colorlinks=true, pdfstartview=FitV, linkcolor=blue, citecolor=blue, urlcolor=blue,pagebackref=false]{hyperref}
\usepackage{mathtools}
\usepackage[shortlabels]{enumitem}

\usepackage[backend=biber,giveninits,maxnames=5, maxalphanames=5,style=alphabetic,isbn=false, date=year, doi=false, eprint= true, url= false, sorting=ynt, sortcites = true]{biblatex}

 \usepackage{authblk}
  
\setkomafont{date}{\large}

\addtokomafont{disposition}{\rmfamily}

\usepackage{microtype}

\usepackage[notcite,notref,color]{showkeys}
\definecolor{labelkey}{gray}{.8}
\definecolor{refkey}{gray}{.8}

\definecolor{darkblue}{rgb}{0,0,0.7} 
\definecolor{darkred}{rgb}{0.9,0.1,0.1}
\definecolor{darkgreen}{rgb}{0,0.5,0}

\renewenvironment{proof}[1][\smallskip\noindent\proofname]{{\smallskip\noindent\bfseries #1. }}{\qed \medskip}

\newtheorem{thm}{Theorem}[section]
\newtheorem{theorem}[thm]{Theorem}
\newtheorem{prop}[thm]{Proposition}
\newtheorem{lem}[thm]{Lemma}
\newtheorem{cor}[thm]{Corollary}
\theoremstyle{definition}

\newtheorem{rem}[thm]{Remark}

\newcommand\norm[1]{\left\lVert#1\right\rVert}

\newcommand\bra[1]{\left({#1}\right)}

\newcommand\abs[1]{\left\lvert#1\right\rvert}

\renewcommand{\le}{\leqslant}
\renewcommand{\ge}{\geqslant}
\renewcommand{\leq}{\leqslant}
\renewcommand{\geq}{\geqslant}
\newcommand{\ls}{\lesssim}
\newcommand{\gs}{\gtrsim}
\renewcommand{\subset}{\subseteq}

\newcommand{\mH}{\mathcal{H}}

\newcommand{\mP}{\mathcal{P}}

\newcommand{\N}{\mathbb{N}}

\newcommand{\1}{\mathbf{1}}
\newcommand{\R}{\mathbb{R}}

\newcommand{\mV}{\mathcal{V}}

\newcommand{\eps}{\varepsilon}

\renewcommand{\div}{\mathop{\rm div}}

\newcommand{\dd}{\, \mathrm{d}}

\DeclareMathOperator{\St}{\mathrm{St}}
\DeclareMathOperator{\Br}{\mathrm{Br}}
\newcommand{\W}{\mathcal{W}}
\newcommand{\dmin}{d_{\mathrm{min}}}

\newcommand{\m}{\mathcal}
\newcommand{\mf}{\mathfrak}

\renewcommand{\varrho}{\rho}

\DeclareMathOperator{\Id}{Id}

\DeclareMathOperator{\dv}{div}

\DeclareMathOperator{\dist}{dist}

\DeclareMathOperator{\diff}{diff}

\numberwithin{equation}{section}

\mathtoolsset{showonlyrefs}

\bibliography{bib.bib}

%% file: 1_Introduction.tex
\section{Introduction}

We continue the mean-field analysis of a microscopic model of rigid  particles  sedimenting in a Stokes flow initiated in \cite{HoferSchubert23}. In \cite{HoferSchubert23} we obtained the macroscopic transport-Stokes system to leading order for small particle inertia.  Going beyond these results, we show in the present paper that the small inertia is accounted for by the mesoscopic Vlasov-Stokes system which consequently yields a better approximation for the microscopic system.

We consider the following dimensionless model of $N$ rigid spherical particles sedimenting in a Stokes flow.
\begin{align}	\label{eq:acceleration}
    &\dot X_i = V_i, \qquad \dot{V}_i = \lambda_N \left(g +  \frac{1}{6 \pi R} \int_{\partial B_i} \sigma[u_N] n \dd  \mathcal{H}^2 \right),\\
\label{eq:fluid.micro}
&\left\{\begin{array}{rl}
		- \Delta u_N + \nabla p_N = 0, ~\dv u_N &=0 \quad \text{in} ~ \R^3 \setminus \bigcup_i B_i, \\
		  u_N(x) &= V_i\quad \text{in} ~ B_i.
\end{array}\right.
\end{align}

Here,  $X_i,V_i\in \R^3$ are the positions and velocities of the particles for $1\le i\le N$. The space occupied by the $i$-th particle is denoted by $B_i := \overline{B_R(X_i)}$, where $R$ is the identical radius of the particles and implicitly depends on $N$. Moreover,  $g \in \mathbb{S}^2$ is the (normalized) gravitational acceleration, $\sigma[u_N] = 2 e u_N - p_N \Id$ is the fluid stress, $e  u_N = \frac 12(\nabla u_N + (\nabla  u_N)^T) $ denotes the symmetric gradient of $u_N$, the outer normal of $\partial B_i$ is denoted by $n$, and $ \mathcal{H}^2$ is the two-dimensional Hausdorff measure. The fluid flow $u_N:\R^3\to \R^3$ is assumed to be at rest at infinity which we encode in the requirement $u_N \in \dot H^1(\R^3)$ to the Stokes problem. Then, there is an associate pressure $p_N \in L^2(\R^3)$ which is unique up to constants.

For details on the modeling we refer to the introduction of \cite{HoferSchubert23} and the references therein.
Note that as in the main part of \cite{HoferSchubert23}, we neglect particle rotations for simplicity. 
Additionally, we assume here $R = 1/(6 \pi N)$, whereas in \cite{HoferSchubert23} we only assumed $NR \sim 1$.  The more stringent  assumption we make here further simplifies the presentation but no technical difficulties arrise from relaxing it to $NR \sim 1$.

We recall that the constant $1/\lambda_N$ accounts for the particle inertia. As in \cite{HoferSchubert23}, we study the limit $N \to \infty$ of the microscopic system \eqref{eq:acceleration}--\eqref{eq:fluid.micro} for small particle inertia $1/\lambda_N \to 0$.
In \cite{HoferSchubert23}
we studied the convergence of the spatial empirical density $\rho_N$ to the solution of the transport-Stokes equation \begin{align}  \label{eq:transport-Stokes}
\left\{\begin{array}{rl}
\partial_t \rho_\ast +(u_\ast+ g) \cdot \nabla \rho_\ast&=0,\\
-\Delta u_\ast+\nabla p_\ast= \rho_\ast g, ~ \dv u_\ast&=0,\\
\rho_\ast(0)&=\rho^0.
\end{array}\right.
\end{align}
We showed that this convergence holds assuming sufficienlty fast convergence of the initial data in the $2$-Wasserstein distance, a (rather restrictive) assumption on the minimal distance between initial particle positions, and assuming that the initial particle velocities do not behave to wildly. Our quantitative convergence result in \cite{HoferSchubert23} includes an error of order $1/\lambda_N$ that reflects that the particle inertia is neglected in the transport-Stokes system \eqref{eq:transport-Stokes}. 
We recall the precise assumptions and the statement from \cite{HoferSchubert23} in Section \ref{sec:main.result}.

In the present paper, we assume additionally, that the initial empitical densities satisfy a uniform bound on the $9$th moment in velocity space. We then show that the microscopic system is well approximated by solutions to the Vlasov-Stokes equation (with $\gamma = 1/(6 \pi)$)
\begin{align} \label{eq:Vlasov.Stokes}
  \left\{\begin{array}{rl}  \partial_t  f_{\lambda_N} + v \cdot \nabla_x  f_{\lambda_N}+ \lambda_N \dv_v((g+6 \pi \gamma( u - v)) f_{\lambda_N}) &= 0, \\
    -\Delta  u_{\lambda_N} + \nabla p = 6 \pi \gamma \int (v- u_{\lambda_N})  f_{\lambda_N} \dd v, \qquad \dv  u &=0, \\
     f_{\lambda_N}(0) &= f^0.
     \end{array}\right.
\end{align}

More precisely, we show that on any finite time interval, the quantitative  error estimate for the $2$-Wasserstein distance  
\begin{align} \label{W_2.intro}
\m W_2(f_N(t),f_{\lambda_N}(t)) \leq C \m W_2(f_N(0),f^0) e^{C t}
\end{align}
holds for sufficiently large $N$ (cf. Theorem~\ref{thm:vlasov}).
Here, $C$ is a constant that only depends on bounds for the initial microscopic particle configuration as well as on norms of the solution $\|f_{\lambda_N}\|.$

In contrast to the error estimate from \cite{HoferSchubert23} where we compare with solutions of the transport-Stokes equation, the error in \eqref{W_2.intro} is better in the sense that it does not contain a term of order $1/\lambda_N$. In this sense, our result is a perturbative derivation of the Vlasov-Stokes equation.

\begin{figure}
    \centering
    \includegraphics[scale=0.5]{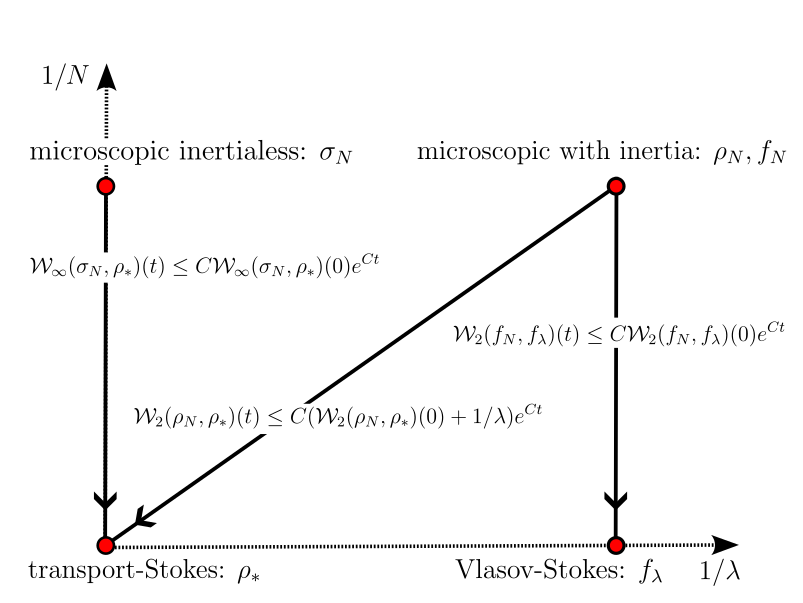}
    \caption{Schematic overview of quantitative mean-field results for sedimenting particles.}
    \label{fig}
\end{figure}

Our result is visualized in Figure \ref{fig}: In the present paper, we estimate the difference $W_2(f_N,f_{\lambda_N})(t)$ when both $\lambda \gg 1, N \gg 1$. In the first part of this series \cite{HoferSchubert23}, we considered the same microscopic model and estimated the difference $W_2(\rho_N,\rho_\ast)(t)$. Previous results (see below) on the derivation of the transport-Stokes equation were restricted to estimating the difference $W_\infty(\sigma_N,\rho_\ast)(t)$, where the empirical particle density $\sigma_N$ corresponds to a microscopic model \emph{without} inertia, i.e. $\lambda = \infty$.

\medskip

Our proof relies on a relative energy estimate. By related arguments, we improve previous results on the well-posedness of the  Vlasov-Stokes equation (cf. Theorem~\ref{thm:existence}) and on the convergence of the Vlasov-Stokes equation towards the transport-Stokes equation (cf. Theorem~\ref{thm:hydrodynamic.limit}). The latter result in particular guarantees that the constant $C$ on the right-hand side of \eqref{W_2.intro} only depends on $f_{\lambda_N}$ through $f^0$ for sufficiently regular $f^0$.

 \subsection{Previous results}

In this subsection we give an overview over previous results on the Vlasov-Stokes equation and related models.
The Vlasov-Stokes equation and variants such as the Vlasov-Navier-Stokes equation and the Vlasov-Fokker-Planck-Navier-Stokes equation (the latter including diffusion for the dispersed phase) have been  proposed, mostly as a model for aerosols, in \cite{ORourke81, CaflischPapanicolau83, Williams18}.
Since then, these models have been extensively studied mathematically.

\subsubsection*{The Cauchy problem and long-time behavior}
Global existence and large-time behavior  of the Vlasov-(instationary)Stokes equation in bounded domains has been considered in  \cite{Hamdache98, Hofer18InertialessLimit, HKP23Vlas_Stok}.  Global existence of weak solutions to the Vlasov-Navier-Stokes equation in two and three dimensions has been shown in \cite{AnoshchenkoMonvel97, BoudinDesvillettesGrandmontMoussa09, Yu13, WangYu15, BoudinGrandmontMoussa17} on the torus, bounded domains and even moving domains. 
Uniqueness in dimension two on the whole space and the torus has been established in \cite{HanKwanMiotMoussaMoyano19}.
Global well-posedness for small data has been shown in \cite{ChoiKwon15} together with a conditional long-time result. 
Unconditional long-time behavior for small data solutions has been established in \cite{HanKwanMoussaMoyano20,HanKwan22,ErtzbischoffHanKwanMoussa21} on the three dimensional torus, bounded domains with absorption boundary conditions, and in the whole space, respectively.
Asymptotic stability of nontrivial stationary solutions in two dimensional pipes has been shown in \cite{GlassHanKwanMoussa18}.
Gevrey regularity of solutions has been studied in \cite{Dechicha23}.

For Vlasov-Fokker-Planck-Navier-Stokes, global well-posedness for classical solutions on  $\mathbb T^3$ close to equilibrium and their long-time behavior has been studied in \cite{GoudonHeMoussaZhang10},
and existence of weak solutions in dimensions two and three on the torus and the whole space has been shown in \cite{ChaeKangLee11}, as well as uniqueness and regularity in dimension two. Existence of strong solutions to the Vlasov-Stokes equation in $\R^3$ is also proved in \cite{ChaeKangLee11}.

With the exception of \cite{Hamdache98,Hofer18InertialessLimit}, all these results concern equations without gravity. In this case, kinetic energy (and entropy in the case with diffusion) of both the fluid and the particles is dissipated. This leads to an energy inequality for weak solutions that is used for the proof of existence of solutions and the study of their long time behavior.
Including gravity adds potential energy which makes 
exploiting the energy-energy-dissipation relation more difficult.
For the Vlasov-Stokes equation in $\R^3$ with gravity in the form of \eqref{eq:Vlasov.Stokes}, global existence of solutions with compactly supported density has been shown in \cite{Hofer18InertialessLimit}. The Vlasov-Navier-Stokes equation with gravity in the half space with absorbtion boundary conditions have been studied in \cite{Ertzbischoff21}, where global existence and long time behavior for small data is shown.

For related results when the fluid is assumed to be compressible and/or inviscid we refer to the introductions of \cite{HanKwanMichel21, ChoiJung21}. Considering monokinetic Vlasov equations leads to pressureless Euler type systems. We refer to the recent work \cite{HuangTangZou23} and the references therein.

\subsubsection*{Hydrodynamic limits}

Formally the transport-Stokes equation \eqref{eq:transport-Stokes} arises from passing to the limit $1/\lambda \to 0$ in the  Vlasov-Stokes equation \eqref{eq:Vlasov.Stokes}. This can be viewed as a hydrodynamic limit: Asymptotically, the kinetic description of the dispersed phase is replaced by a macroscopic PDE for the spatial particle density.

Such hydrodynamic limits are very important from the applied point of view because they represent a model reduction. Hydrodynamic limits for simplified versions of the Vlasov-Stokes equation have been studied in \cite{Jab00,Gou01}.
In the seminal works \cite{GoudonJabinVasseur04a, GoudonJabinVasseur04b}, two hydrodynamic limits for the Vlasov-Fokker-Planck-Navier-Stokes equation without gravity have been studied. The authors call these two limits \emph{light particle regime} and \emph{fine particle regime}. 
Roughly speaking they correspond to the limits $1/\lambda \to 0$ for fixed $\gamma$, and $\gamma \to \infty$ for fixed $\lambda$ in \eqref{eq:Vlasov.Stokes}, but the equation involves additional terms accounting for the fluid inertia and the particle diffusion.
In both cases, the limit consists of an advection-diffusion equation for the particle density coupled to the Navier-Stokes equation. In the fine particle regime, the Navier-Stokes equation become \emph{inhomogeneous}, though, as the particles endow the fluid with additional inertia, whereas in the light particle regime, the particles do not affect the Navier-Stokes equation.
Corresponding hydrodynamic limits have been established in \cite{HanKwanMichel21} for the Vlasov-Navier-Stokes equation without gravity.

In  \cite{Hofer18InertialessLimit}, the limit $1/\lambda \to 0$ has been studied for the Vlasov-Stokes equation \eqref{eq:Vlasov.Stokes}, and convergence to the transport-Stokes equation \eqref{eq:transport-Stokes} has been proved.
In \cite{Ertzbischoff22}, the corresponding limit for the Vlasov-Navier-Stokes equation with gravity has been established for small data on the half-space with absorption boundary conditions.

We refer to \cite{HanKwanMichel21}
for more references on related hydrodynamic limits.

\subsubsection*{Derivation of the Vlasov-Stokes equation}

In the spirit of Hilbert's $6$th problem, considerable effort has been invested towards a rigorous derivation of the Vlasov-(Navier-)Stokes equation \eqref{eq:Vlasov.Stokes} starting from a  microscopic description of the suspension like \eqref{eq:acceleration}--\eqref{eq:fluid.micro}. However, the results are presently only fragmentary. 

The full Vlasov-Fokker-Planck-Navier-Stokes equation has been derived in \cite{FlandoliLeocataRicci19, FlandoliLeocataRicci21} but starting from an intermediate description of the microscpoic model. More precisely,
 in \cite{FlandoliLeocataRicci19, FlandoliLeocataRicci21} the authors consider a system where the the fluid solves the Navier-Stokes equation in the whole space instead of solving it only within the fluid domain with appropriate boundary conditions at the particles.
To account for the effect of the particles, a smeared out version of the friction force $u(X_i) - V_i$ is prescribed to act both on the fluid at position $X_i$ and on the $i$-th particle.
A different approach towards the derivation of the Vlasov-Navier-Stokes has been initiated in \cite{BernardDesvillettesGolseRicci17, BernardDesvillettesGolseRicci18} where the starting point consist in two coupled Boltzman equations, one  for the particle and and one for the fluid phase.

Starting from \eqref{eq:fluid.micro}, the derivation of the Brinkman equation, i.e. the fluid equation in \eqref{eq:Vlasov.Stokes},  when the particle positions are given, is by now classical. In the seminal paper \cite{Allaire90a} the Brinkman equation for periodic particle configuration with zero particle velocities are derived through homogenization methods adapted from \cite{CioranescuMurat97, Tartar80}. In \cite{DesvillettesGolseRicci08} more general particle configurations are considered with non-zero particle velocities.  In \cite{HillairetMoussaSueur17} effects of particle rotations are included,  and in \cite{GiuntiHofer18, CarrapatosoHillairet18, HoeferJansen20} the assumptions on the particle configurations have been further relaxed, allowing for probabilistic statements and even the study of fluctuations of the fluid velocity around the homogenization limit. We also mention that the instationary Navier-Stokes-Brinkman equation has been derived in a corresponding homogenization limit in \cite{Feireisl2016} and combined homogenization and inviscid limits have been studied in \cite{Hofer22}. We refer to the introduction of \cite{Hofer22} for
a more detailed overview over related homogenization results.
For a discussion of homogenization results when different boundary conditions at the interfaces with the particles are considered, we refer to the introduction of \cite{HoferSchubert23}.

When particle inertia is neglected in the microscopic description (that is, $\dot V_i$ is replaced by $0$ in \eqref{eq:acceleration}), the transport-Stokes equation (and perturbations of it) has been derived in \cite{Hofer18MeanField, Mecherbet19, Hofer&Schubert, Duerinckx23}. We refer to the introduction of \cite{HoferSchubert23} for more details  on these results.

\subsection{Key novelties and elements of the proof}

\label{sec:strategy}

We comment here on the proof of the perturbative derivation of the Vlasov-Stokes equation stated in Theorem \ref{thm:vlasov}.
For comments on the proofs of the improved well-posedness result and the hydrodynamic limit, we refer to Sections \ref{sec:result.well-posedness} and \ref{sec:result.hydrodynamic.limit}, respectively.

To our knowledge, we present here the first rigorous derivation of a Vlasov-fluid system starting from a microscopic description of the system through an ODE-system for the particles coupled to a fluid PDE with suitable boundary conditions at the particle interface. 
We emphasize though that our result has several shortcomings which we are presently unable to overcome:
\begin{itemize}
    \item The fluid inertia is neglected and thus the fluid satisfies the Stokes equation instead of the Navier-Stokes equation. 
    \item The initial particle configurations have to satisfy a minimal distance condition that is typically not satisfied  for reasonable random  particle configurations. 
    \item The result is of perturbative nature in the sense that the inertia  has to satisfy $1/\lambda_N \to 0$. As shown in the first part \cite{HoferSchubert23}, the spatial particle density $\rho_N$ thus converges to the solution of the transport-Stokes equation \eqref{eq:transport-Stokes}.
    We show in the present paper, that the particle density is better approximated by the solution of the Vlasov-Stokes equation \eqref{eq:Vlasov.Stokes}. 
\end{itemize}
The results in \cite{Hofer18MeanField, Mecherbet19, Hofer&Schubert, Duerinckx23, HoferSchubert23} regarding the derivation of the transport-Stokes equation also suffer from the first two deficits. 

Regarding the first and third deficit, we emphasize that assuming the fluid and particle inertia to be very small is reasonable for many applications (cf. the introduction of \cite{Hofer18MeanField} for a more detailed discussion).

\paragraph*{Techniques and difficulties for second-order mean-field limits with singular interaction.}
As with many works on mean-field limits, our proof is based on a stability estimate for the limit equation that a priori only works for sufficiently regular particle densities. Typically, one of the two solutions can have quite low regularity, as long as the other is regular enough (compare to weak-strong uniqueness results). In particular the low regularity for one of the solutions eventually allows to connect to the microscopic system as these are often (approximate) distributional solutions of the limit equation.

In this spirit, the core of our proof is a stability estimate in the Wasserstein metric based on optimal transport arguments.
Such stability estimates for Vlasov equations go back to Dobrushin (for regular kernels) \cite{Dobroushin79} and have been adapted in a celebrated result by Loeper \cite{Loeper06} for the Vlasov-Poisson and two dimensional Euler equation.
Further adaptations of such stability estimates have been used for the derivation of Vlasov equations  from $N$-particle systems with singular binary interaction in \cite{HaurayJabin2007, HaurayJabin2015, Hauray09}. 
Unfortunately, these results  only allow for  singular interactions like $|x|^{-\alpha}$ for $\alpha$ not too big: For first-order mean-field systems the results in \cite{Hauray09} allow for $\alpha < d-1$, where $d$ is the space dimension. \cite{HaurayJabin2007, HaurayJabin2015} provide second-order mean-field limits if $\alpha < 1$. The particle interaction in \eqref{eq:acceleration} is not a simple binary interaction through a given kernel. However, the corresponding singularity is at least formally given by $\alpha =1 $. Indeed, as we have shown in \cite{HoferSchubert23}, a good approximation for \eqref{eq:acceleration} is given by the Stokes type law (see Lemma \eqref{le:force.representation})
\begin{align}
    \dot V_i \approx \lambda_N \left( g + (u)_i - V_i \right) 
\end{align}
where $(u)_i$ is a local average of $u$ around the $i$-th particle which itself can be approximated by the implicit relation (cf. \cite[Lemma 5.4]{HoferSchubert23})
\begin{align}
    (u)_i    \approx \sum_{j \neq i} \Phi(X_i - X_j) (V_j -  (u)_j). 
\end{align}
Here $\Phi$ is the Oseen tensor (see \eqref{Oseen}), which is $-1$-homogeneous.

This singular (and implicit) nature of the interaction is one of the the key difficulties in the derivation of the Vlasov-Stokes equation: Even when we ignore the problem related to the implicit nature of the interaction, the singularity is (borderline) too strong to fall into the framework of \cite{HaurayJabin2007, HaurayJabin2015} for the derivation of second-order mean-field limits. Despite recent progress for first-order mean-field limits and second-order mean-field limits with diffusion (see e.g. \cite{Serfaty20, BreschJabinWang20, BreschJabinSoler22}), to the best of our knowledge, \cite{HaurayJabin2007, HaurayJabin2015} remain the only results regarding second-order mean-field limits without diffusion and with singular interaction without cutoff.

We do not overcome the conceptual difficulties related to the singular interaction in our present paper, but the perturbative nature of our result circumvents them:  we use that we already know from the first part \cite{HoferSchubert23} that the minimal (spatial) inter-particle distance only shrinks exponentially in time
and that the spatial particle density is close to some continuous density $\rho_\ast \in L^q(\R^3)$ (which is the solution of the transport-Stokes equation \eqref{eq:transport-Stokes}).
Compared to \cite{HaurayJabin2007, HaurayJabin2015}, we consequently do not need to get such control on the microscopic particle configuration  through a buckling argument from the solution of the target system. 

\paragraph*{Uniform stability estimate for the Vlasov-Stokes equation.} By considering the perturbative regime $1/\lambda_N\to 0$ we have to pay the prize that the target system involves this singular limit. This means that we need uniform stability estimates in this singular limit which is nontrivial even for sufficiently smooth solutions. 
In particular, we cannot directly adapt the existing stability estimates for the Vlasov-(Navier-)Stokes equation in the literature. For example a stability estimate for the two dimensional Vlasov-Navier-Stokes equation is used in \cite{HanKwanMiotMoussaMoyano19} to obtain uniqueness of solutions. We point out that the main difficulties in \cite{HanKwanMiotMoussaMoyano19} are related to the  fluid inertia, which we do not face as we consider the Vlasov-Stokes equation.
In particular, we use a similar stability estimate to prove well-posedness and stability
of solutions to the Vlasov-Stokes equation in Section \ref{sec:well-posedness}. However, this stability estimate becomes useless as $1/\lambda_N \to 0$.

We show a stability estimate which is uniform in $\lambda$ by exploiting more carefully the  dissipation of kinetic energy by the Vlasov-Stokes equation \eqref{eq:Vlasov.Stokes}.
As in \cite{HoferSchubert23}, it is therefore crucial that, rather than working with the $\infty$-Wasserstein metric as in \cite{HaurayJabin2007, HaurayJabin2015}, we work with  the $2$-Wasserstein metric which is naturally linked to the kinetic energy of the particles. 
In particular, our stability estimate is obtained through a modulated energy argument that allows us to use the structural properties of the equation. Related modulated energy arguments have been successfully used in a numerous works to study kinetic equations, including their long-time behavior, hydrodynamic limits as well as their derivation. We refer to related modulated energy arguments that exploit the dissipative structure of the Vlasov-(Navier-)Stokes equation  in \cite{ChoiKwon15, HanKwanMoussaMoyano20, HanKwan22, Ertzbischoff21, ErtzbischoffHanKwanMoussa21, Ertzbischoff22}.

For the argument, specific to the Vlasov-Stokes equation for very large $\lambda$, we combine here ideas from \cite{Hofer18InertialessLimit} and \cite{HoferSchubert23}. In \cite{HoferSchubert23} we used  dissipation mechanisms on the microscopic system  to compare with a microscopic inertialess system. There, one key estimate has been the uniform coercivity of the so called resistance matrix to quantify the energy dissipation due to friction between the particles and the fluid, which is given in terms of the Dirichlet energy of the fluid velocity.
In \cite{Hofer18InertialessLimit}, the hydrodynamic limit $1/\lambda \to 0$ passing from the Vlasov-Stokes equation \eqref{eq:Vlasov.Stokes} to the transport-Stokes equation \eqref{eq:transport-Stokes} has been established. The proof relies on the estimate of the energy dissipation for the Vlasov-Stokes equation which is the macroscopic analogue of the coercivity of the resistance matrix. 

\paragraph*{Dealing with the implicit nature of the microscopic interaction.}
When adapting the stability argument to the microscopic empirical particle density given as the solution of \eqref{eq:acceleration}--\eqref{eq:fluid.micro}, another difficulty arises through the implicit nature of the interaction on the microscopic level: In contrast to standard mean-field problems (with friction), the microscopic force is not given as $F_i = \frac 1 N \sum_{j \neq i} K(X_j - X_i) - V_i$ for some kernel $K$, neither is the macroscopic force $u - v$ given in terms of a convolution (at least not through a convolution kernel that does not depend on the solution itself). Such convolutions with a singular kernel are very convenient when one already controls singular discrete sums of the particle configurations, as we do due the results in \cite{HoferSchubert23}.
Instead, $u$ is the solution of  the Brinkman equation, which does not even make sense for the empirical measure $f_N$.
To overcome this issue, we build on observations in \cite{HoferSchubert23} that relate the force $F_i$ to $u_N - V_i$. Moreover, in order to give a meaning to the microscopic Brinkman equation, we smear out the spatial density $\rho_N$ on an intermediate scale $d$, with $R \ll d \ll 1$. This allows us to eventually express $u_N-u$ through a convolution up to errors which vanish as $N\to \infty$.

\subsection{Structure of the remainder of the paper}

In Section \ref{sec:main.result}, we state the three main results of our paper: Subsection \ref{sec:result.well-posedness} contains the (improved) well-posedness result for the Vlasov-Stokes equation, in Subsection \ref{sec:result.hydrodynamic.limit}, we state the hydrodynamic limit in the regime of vanishing inertia and in Subsection \ref{sec:result.derivation} we present the result on the perturbative derivation of the Vlasov-Stokes equation.
Finally we collect the notation that is used in the paper in Section \ref{subsec:notation}.

In Section \ref{sec:preliminary}, we collect some estimates on the Brinkman equation that are used throughout the paper. 
In the same order as stated in Section \ref{sec:main.result}, we devote Sections \ref{sec:well-posedness}, \ref{sec:hydrodynamic.limit} and \ref{sec:derivation} to the proofs of the three main results. 
Moreover, in Subsection \ref{sec:stability.macro} we provide and show an improved stability estimate for the Vlasov-Stokes equation for  large $\lambda$ which is the starting point for the derivation of the  Vlasov-Stokes equation.

%% file: 2_MainResults.tex
\section{Main results}
\label{sec:main.result}

We consider the Vlasov-Stokes equation \eqref{eq:Vlasov.Stokes} with $\gamma=1/(6\pi)$:
\begin{align} \label{eq:Vlasov.Stokes.gamma=1}
\left\{\begin{array}{rl}
    \partial_t  f + v \cdot \nabla_x  f + \lambda \dv_v((g+ u - v) f) &= 0, \\
    -\Delta  u + \nabla p = \int (v- u)  f \dd v, \qquad \dv  u &=0, \\
     f(0) &= f^0.
     \end{array}\right.
\end{align}
All results extend in a straightforward way to general $\gamma \in (0,\infty)$.

\subsection{Well-posedness and stability of the Vlasov-Stokes equation}
\label{sec:result.well-posedness}
For $k \geq 0$, we denote the $k$-th moment of a function $h \colon \R^3 \times \R^3 \to \R_+$ by
\begin{align} \label{def:M_k}
  M_k[h] :=   \int |v|^k h\dd v \dd x.
\end{align}
We denote by $\W_2(\cdot,\cdot)$ the $2$-Wasserstein distance and by $ C((0,T); w - \mathcal P(\R^3 \times \R^3))$ the space of functions $h : (0,T) \to \mathcal P(\R^3 \times \R^3)$ which are continuous in the $\W_2$ topology. 

\begin{theorem} \label{thm:existence}
    Let $K > 9 $ and 
    let $f^0 \in \mP(\R^3 \times \R^3) \cap L^\infty(\R^3 \times \R^3)$ with $M_K[f^0] < \infty$.
    For every $T > 0$ there exists a unique weak solution $(f,u)$ to \eqref{eq:Vlasov.Stokes.gamma=1} with $f \in L^\infty((0,T) \times \R^3 \times \R^3) \cap C((0,T); w - \mathcal P(\R^3 \times \R^3))$, $u \in L^\infty(0,T;W^{1,\infty}(\R^3)) \cap C((0,T) \times \R^3)$, and
    \begin{align} \label{moments.thm}
        \sup_{t \in (0,T)} M_K(f(t)) < \infty.
    \end{align}
    Moreover, if $h^0  \in L^\infty(\R^3 \times \R^3)$ is another initial datum with $M_K[h^0]<\infty$, and with corresponding solution $h$, then, for all $T > 0$, 
    \begin{align} \label{eq:stability}
        \sup_{t \leq T} \W_2(f(t),h(t)) \leq C \W_2(f^0,h^0),
    \end{align}
    where the constant $C$ depends only on $T$, $\lambda$, $K$, $\|f^0\|_{L^\infty(\R^3)}$, $\|h^0\|_{L^\infty(\R^3)}$, $M_K[f^0]$ and $M_K[h^0]$. 
\end{theorem}
\begin{rem}
\begin{enumerate}[(i)]
\item We say that the pair $(f,u)$ is a weak solution of \eqref{eq:Vlasov.Stokes} if $f$ satisfies the first equation in the sense of distributions and $u$ satisfies the second equation in the usual weak sense for almost every time.
    \item Compared to \cite{Hofer18InertialessLimit}, where well-posedness of strong solutions to the Vlasov-Stokes equation have been obtained, Theorem \ref{thm:existence} has considerably lower requirements on the initial datum $f^0$. More precisely, \cite{Hofer18InertialessLimit} only covers compactly supported $f^0 \in W^{1,\infty}(\R^3 \times \R^3)$. 
     \item The assumption $M_K[f^0] < \infty$ for some $K > 9$ seems to be optimal to guarantee that $u \in L^\infty(0,T;W^{1,\infty}(\R^3))$ (cf. Lemma~\ref{lem:u.Lipschitz} and Lemma~\ref{lem:moments}). This allows for a classical Lagrangian approach. 
     \item A related assumption, namely $M_K[f^0 + |\nabla_x f^0| + |\nabla_v f^0|] < \infty$ for some $K > 9$  appears in \cite{HKP23Vlas_Stok}, where well-posedness for strong solutions of the Vlasov-(transient-)Stokes equation is shown on the three dimensional torus.
\end{enumerate}
\end{rem}

For the proof of Theorem~\ref{thm:existence} we use the Banach fixed point theorem in a suitable subset of $ C((0,T); w - \mathcal P(\R^3 \times \R^3)$ endowed with the $2$-Wasserstein metric (in supremum with respect to time).
The contraction property of an appropriate fixed point formulation then essentially boils down to a stability estimate in the $2$-Wasserstein metric in the spirit of the classical argument by Dobrushin \cite{Dobroushin79}.
A crucial ingredient to get a global result  is the classical energy equality for solutions of the Vlasov-Stokes equation
\begin{align} \label{energy.identity}
    \frac 1 2\dot  M_2[f] + \lambda \|\nabla u\|_2^2 + \lambda \int |u-v|^2 \dd f(x,v) = \lambda \int v \cdot g  \dd f(x,v) \leq \lambda M_2[f]^{\frac 1 2}.
\end{align}
It allows to propagate first bounds on the kinetic energy $M_2[f]$ and then, through a bootstrap argument, the bound on the highest moment $M_K[f]$.

\subsection{Hydrodynamic limit of the Vlasov-Stokes equation in the inertialess regime}
\label{sec:result.hydrodynamic.limit}

For $h \colon \R^3 \times \R^3 \to \R_+$, we introduce the notation 
\begin{align}
    \rho[h] = \int h \dd v.
\end{align}

\begin{theorem} \label{thm:hydrodynamic.limit}
    Let $K > 9 $, $T>0$ and 
    let $f^0 \in \mP(\R^3 \times \R^3) \cap L^\infty(\R^3 \times \R^3)$ with $M_K[f^0] < \infty$.
    Let $(f_\lambda,u_\lambda)$ be the unique solution 
    to \eqref{eq:Vlasov.Stokes.gamma=1} from Theorem \ref{thm:existence}.
    Moreover, let $(\rho_\ast, u_\ast)$ be the unique solution of the transport-Stokes equation \eqref{eq:transport-Stokes} with initial datum $\rho^0 = \rho[f^0]$.
    Then, there exists  $\eps_0>0$  depending only monotonically on $K$,  $M_K[f^0]$, $\|f^0\|_{L^\infty(\R^3)}$  and $T$ such that the following holds.
    If at least one of the three conditions
    \begin{enumerate}
        \item \label{cond:L^1L^infty} $f^0 \in L^1_v(\R^3;L^\infty_x(\R^3))$,
        \item \label{cond:f.small} $\|f^0\|_{L^\infty(\R^3 \times \R^3)} \leq \eps_0$, 
        \item \label{cond.well.prepared} $\int_{\R^3 \times \R^3} |v - g - u_\ast(0,x)|^2 \dd f^0(x,v) \leq \eps_0$,
    \end{enumerate}
     is satisfied\, then there exist $C<\infty ,c>0$ and $\lambda_0< \infty$,  depending only monotonically on $K$,  $M_K[f^0]$, $\|f^0\|_{L^\infty(\R^3)}$, and on $T$, and in addition on $\|f^0\|_{L^1_v(\R^3;L^\infty_x(\R^3))}$ in case of Condition \ref{cond:L^1L^infty}, such that  for all $t \leq T$ and all $\lambda \geq \lambda_0$ and all $x \in \R^3$ one has
    \begin{align}
        &\W_2(\rho_\ast(t), \rho[f_\lambda(t)]) \leq \frac C \lambda , \label{W.rho.hydrodynamic}\\
        &\W_2(f_\lambda(t), \rho_\ast(t) \otimes \delta_{g + u_\ast(t)}) + \|u_\ast - u_\lambda\|_{L^2(B_1(x))}\\
        &\qquad\qquad\qquad\leq  C  e^{-c \lambda t} \left(\int |v - g - u_\ast(0,x)|^2 \dd f^0(x,v)\right)^{\frac 1 2}  + \frac C \lambda , \label{W.f.hydrodynamic}\\
        &M_K(f_\lambda(t)) + \|\rho[f_\lambda(t)]\|_{L^{1 + K/3}(\R^3)} + \|u_\lambda(t)\|_{W^{1,\infty}(\R^3)}  \leq C. \label{uniform.bounds.thm}
    \end{align}
\end{theorem}
\begin{rem} \label{rem:moments.main}
\begin{enumerate}[(i)]
\item The well-posedness of the transport-Stokes equation for initial data $\rho^0 \in L^p(\R^3)$
for $p \geq 3$ has been shown in \cite[Theorem 2.1]{MecherbetSueur22}.
    By Lemma \ref{lem:moments}, the moment condition $M_K[f^0] < \infty$ for some $K > 9$ implies $\rho^0 \in L^{1+K/3}(\R^3)$. 
 \item Analogously to the well-posedness result, Theorem \ref{thm:existence}, this result improves on \cite{Hofer18InertialessLimit}, where the same hydrodynamic limit has been obtained for compactly supported initial data $f^0 \in W^{1,\infty}(\R^3 \times \R^3)$.
\end{enumerate}
\end{rem}
The proof is based on a modulated energy argument, monitoring the quantities
    \begin{align}
    S(t) &\coloneqq \frac 1 2 \int |v - g -w_\lambda|^2 \dd f_\lambda(t,v,x) = \frac 1 2 \int |V_\lambda - g -w_\lambda \circ X_\lambda|^2 \dd f^0, \\
    Z(t)  &\coloneqq  \frac 1 2 \int |X_\lambda - X_\ast|^2 \dd f^0,
\end{align}
where $(X_\lambda,V_\lambda)$ are the characteristics associated with $f_\lambda$, $X_\ast$ are the characteristics associated with $\rho_\ast$ and 
$w_\lambda$ is the solution to the problem
\begin{align}
    - \Delta w_\lambda + \nabla p = \varrho_\lambda g, \quad \dv w_\lambda = 0.
\end{align}
 The main difficulty in the argument is the proof of the uniform estimates \eqref{uniform.bounds.thm}. Indeed, through the energy equality \eqref{energy.identity}, one a priori only obtains bounds which blow up as $\lambda \to \infty$. Moreover, the $L^\infty$ norm of the phase space density $f$ must blow up because of the convergence to a Dirac with respect to the velocities.
To overcome these issues, one needs to make rigorous that this blow-up is only with respect to the velocity variable. More precisely, one can show that the map		$v \mapsto V_\lambda(0,t,x,v)$
is bi-Lipschitz and that its inverse $W_\lambda(t,x,w)$ is well-defined with 
$\det \nabla_w W_\lambda(t,x,w) \approx e^{ -3 \lambda t}$. By a change of variables one then gets
	\begin{equation}
		\label{eq:rhoByW.intro}
		\rho[f_\lambda](t,x) = \int e^{ 3 \lambda t} f^0(X(0,t,x,W(t,x,w)),w) \det \nabla_w W_\lambda(t,x,w)  \dd w.
	\end{equation}
These properties have been shown in \cite{Hofer18InertialessLimit} and similar arguments have been used for example in \cite{BardosDegond85, HanKwanMichel21, Ertzbischoff22}.
If $f^0 \in L^1_v(\R^3;L^\infty_x(\R^3))$, which is assumed in \cite{Hofer18InertialessLimit, HanKwanMichel21, Ertzbischoff22}, then this yields a uniform estimate on the spatial density $\rho[f_\lambda]$ in $L^\infty(\R^3)$.

 By refining the arguments, we can improve on this assumption through providing alternative conditions on the initial data.
A buckling argument shows that \eqref{uniform.bounds.thm} holds  after an initial layer of size $1/\lambda$ \emph{if} it holds during the initial layer.
During this initial layer the phase space density converges to a Dirac measure  with respect to the velocities. This is reflected by estimate \eqref{W.f.hydrodynamic} whose right-hand side is only small after the initial layer. 
The three conditions in the statement are used in the proof to control the left-hand side in \eqref{uniform.bounds.thm} during this initial layer. More precisely:
\begin{itemize}
    \item 
In case of Condition 1, the desired estimate follow from \eqref{eq:rhoByW.intro} as in \cite{Hofer18InertialessLimit, HanKwanMichel21, Ertzbischoff22}.

\item If Condition 2 holds, one can show that the energy identity \eqref{energy.identity} alone suffices to also get uniform control on $M_2[h]$ for times of order $1$, because the dissipation (the terms $\lambda \|\nabla u\|_2^2$,  $\lambda \int |u-v|^2 \dd f(v,x)$) dominates the injection of energy through gravity.

\item  Condition 3 means that the data are sufficiently well-prepared that the buckling argument works even in the initial layer because the right-hand side of \eqref{W.f.hydrodynamic} is small even for small times $t$.
\end{itemize}

Without these three conditions, we do not know whether the statement remains true. However, additional alternative conditions could be formulated. For example, in view of the convergence result in \cite{HoferSchubert23} (see Theorem \ref{th:diagonal} below), we expect  a suitable continuous version of conditions \eqref{ass:V.Lipschitz}--\eqref{ass:V.norms} to be sufficient, too.

\subsection{Perturbative derivation of the Vlasov-Stokes equation}
\label{sec:result.derivation}
As in \cite{HoferSchubert23} we consider  $X_i^0, V_i^0$, $1 \leq i \leq N$ (implicitly depending on $N$), that satisfy the no-touch condition (which will be implied by conditions \eqref{ass:gamma}--\eqref{ass:V.norms} below for large enough $N$)
\begin{align}  \label{nonoverlapping}
       |X_i^0 - X_j^0| > 2R,
\end{align}
such that there exists a unique solution $u_N$, $X_i, V_i$ to \eqref{eq:acceleration}--\eqref{eq:fluid.micro}.
With the solution we associate the empirical density and its first marginal
\begin{align}
    f_N(t) = \frac 1 N \sum_i \delta_{X_i(t)} \otimes \delta_{V_i(t)}, &&     \rho_N(t) = \frac 1 N \sum_i \delta_{X_i(t)}.
\end{align}
Moreover, we introduce the minimal distance between the particles (which implicitly depends on $N$) 
\begin{align}
    \dmin(t) = \min_{i \neq j} |X_i(t) - X_j(t)|.
\end{align}

We will work under the following assumptions.
\begin{align} 
    \label{ass:gamma} \tag{H1}
     & NR  = \frac 1 {6\pi}, \\[5pt]
     	&\exists q> 4, \rho^0 \in \mathcal P(\R^3) \cap L^q(\R^3): \\\label{ass:W.lambda}  \tag{H2}
    &\qquad \qquad \lim_{N \to \infty} \bra{\W_2(\rho^0, \rho_N(0))  + \frac 1 {\lambda_N}}^{\frac 2 {3(3 + 2q')}}\left(1 + \frac 1 {N^{1/3}\dmin(0)}\right) = 0, \quad  \\[10pt]
    \label{ass:V.Lipschitz} \tag{H3}
    &\forall N \in \N: ~ \forall\, 1 \leq i,j \leq N: ~ |V_i^0 - V_j^0| \leq \frac 12  \lambda_N |X_i^0 - X_j^0|,   \\[7pt] 
    \label{ass:V.norms}\tag{H4}
    &\exists C_V > 0: ~ \forall N \in \N: ~ \frac 1 N \sum_i  |V_i^0|^9 + \frac{1}{\lambda_N} \sup_i |V^0_i|_\infty \leq C_V,
\end{align} 
where \noeqref{ass:W.lambda} $q'=\frac{q-1}{q}$ is the H\"older dual of $q$. These assumptions are largely the same as in \cite{HoferSchubert23} 
(and we refer to \cite[Remark 1.3]{HoferSchubert23} for comments on them).
We comment here on the three differences, which slightly strengthen the assumptions compared to \cite{HoferSchubert23}: 
\begin{enumerate}[(i)]
    \item In \cite[(H1)]{HoferSchubert23})  we assume $\gamma_N \coloneqq NR$ to be of order one. The strengthening of this assumption in \eqref{ass:gamma} is solely made for the sake of simpler statements and notation.
    \item In \eqref{ass:V.norms}, we replaced control of the kinetic energy of the particles, i.e. $\frac 1 N \sum_i |V_i^0|^2$ by the (stronger) bound on the $9$-th moment. This should be compared to the assumption $M_K[f^0] < \infty$ for some $K > 9$ in Theorem \ref{thm:existence}.
    \item In \cite[(H2)]{HoferSchubert23}, only $q >3$. Again this should be compared with the assumption $M_K[f^0] < \infty$ in Theorem \ref{thm:existence} and Remark \ref{rem:moments.main}. Moreover, we decreased the exponent from $\tfrac{1}{3 + 2q'}$ to $\tfrac{2}{3(3+2q')}$ which further strengthens the assumption. This is needed for additional control on singular sums (cf. \eqref{eq:S_est}).
\end{enumerate}

For future reference, we combine \eqref{ass:W.lambda} with \cite[Eq. (1.16)]{HoferSchubert23}
to see that
\begin{align} \label{dmin.threshold}
    \dmin \gg  N^{-\frac{59}{153}}. 
\end{align}
In particular, by \eqref{ass:gamma}, and for $N$ large enough
\begin{align} \label{theta}
        \dmin(0) \gg   R
\end{align}
such that the particles do not overlap initially for $N$ sufficiently large.

We recall the main result of \cite{HoferSchubert23} (in a slightly weaker form due to the strengthening of the assumptions).
\begin{theorem}[{\cite[Theorem 1.1]{HoferSchubert23}}] \label{th:diagonal}
For $N \in \N$ assume that $\lambda_N$,  $X_i^0, V_i^0$  satisfy \eqref{ass:gamma}--\eqref{ass:V.norms}.
Let $(\rho_\ast,u_\ast) \in L^\infty((0,T);L^q(\R^3))\times L^\infty((0,T);\dot H^1(\R^3))$ be the unique weak solution of \eqref{eq:transport-Stokes} associated to the initial value $\rho^0$. Then, there exists $C > 0$ depending only on the constants from \eqref{ass:gamma}--\eqref{ass:V.norms} and on $\|\rho^0\|_{L^q(\R^3)}$ such that the following holds. For all $T \geq 0$ there exists $N_0$ with the property that for all $N \geq N_0$ and all $t \leq T$ it holds that
\begin{align} \label{W_2.simplified}
    \W_2 (\rho_N(t),\rho_\ast(t)) &\leq C \left(\W_2 (\rho_N(0),\rho^0) +\frac{1}{\lambda_N}\right)e^{Ct}. 
\end{align}
Moreover the minimal distance is estimated by
\begin{align}
    \dmin(t) &\geq \frac{\dmin(0)} C  e^{-C t}, \label{dmin.thm}
\end{align}
and
\begin{align}\label{eq:w2_phase}
\W_2(f_N(t),\rho_\ast(t)&\otimes\delta_{(u_\ast(t)+g)})\\
&\le C \left(\W_2 (f_N(0),\rho^0\otimes\delta_{(u_\ast(0)+g)}) +\frac{|V-\tilde V|_2(0)}{\sqrt N} e^{- \frac{\lambda_N t}{C}}+\frac{1}{\lambda_N}\right)e^{Ct},
\end{align}
where $\tilde V$ is the instantaneous inertialess velocity of the particles associated to their momentary position (cf. \cite[eq. (1.10)]{HoferSchubert23}.
\end{theorem}

Since $1/\lambda_N$  is the order of the (microscopic) particle inertia which is neglected in the macroscopic transport-Stokes equation \eqref{eq:transport-Stokes}, the error of order $1/\lambda_N$ in the above result is unavoidable.
To improve on this error, we need to compare with the Vlasov-Stokes equation \eqref{eq:Vlasov.Stokes.gamma=1}.

\begin{theorem}\label{thm:vlasov}
For $N \in \N$ assume that $\lambda_N$, $X_i^0$, $V_i^0$  satisfy \eqref{ass:gamma}--\eqref{ass:V.norms}.
    Let $K>9$, $T> 0$, $q = 1 + K/3$, $f^0 \in \mP(\R^3 \times \R^3) \cap L^\infty(\R^3 \times \R^3)$ with $M_K[f^0] < \infty$ and let $(f_{\ast,\lambda_N},u_{\ast,\lambda_N})$ be the unique  solution on $[0,T]$ to the Vlasov-Stokes equation \eqref{eq:Vlasov.Stokes.gamma=1} with $\lambda = \lambda_N$ and with initial datum $f^0$, provided by Theorem \ref{thm:existence}. 
    Then, there exists $C>0$ depending only on the constants from \eqref{ass:gamma}--\eqref{ass:V.norms} and on $\sup_{t \in [0,T]}\|\rho[f_{\ast,\lambda_N}](t)\|_{L^q(\R^3)}$ such that for all $T \geq 0$ there exists $N_0$ such that for all $N \geq N_0$, all $t \leq T$ and all $x \in \R^3$ it holds that
    \begin{align}\label{eq:est_main}
    \W_2(f_N(t),f_{\ast,\lambda_N}(t))+\norm{u_{\ast,\lambda_N}(t)-u_N(t)}_{L^2(B_1(x))} &\leq C \W_2(f_N(0),f^0) e^{C t}. 
    \end{align}

\end{theorem}
\begin{rem}
    \begin{enumerate}[(i)]
    \item  A priori, the quantity  $\sup_{t \in [0,T]}\|\rho_{\ast,\lambda_N}(t)\|_{L^q(\R^3)}$ that the constant $C$ depends on, itself depends on $\lambda_N$. However, \eqref{uniform.bounds.thm} in Theorem \ref{thm:hydrodynamic.limit} implies a control, uniformly in $N$, under some additional regularity or well-preparedness assumptions.
    \item An estimate on $\W_2(f_N(t),f(t))$, that is sharp at $t=0$, follows from the proof. For this, one has to use the corresponding sharp at $t=0$ estimates from Lemma~\ref{l:ODE}. 
    \item The case of a monokinetic limit density $\rho\otimes \delta_V$ is not included in Theorem~\ref{thm:vlasov} since we require $f$ to be absolutely continuous with respect to the Lebesgue measure. However, by some adaptations of the proof, it is also possible to derive estimates in the monokinetic setting. 
    \end{enumerate}
\end{rem}

\subsection{Notation}\label{subsec:notation}

We will use the following notation throughout. We also recall some previously introduced notation to collect everything in one place.

\begin{itemize}
    \item Throughout proofs, we will write $A\ls B$ to mean $A\le CB$ for some constant $C<\infty$ that only depends on the quantities that have been specified in the corresponding statements.

    \item We recall that $g \in \mathbb{S}^2$ denotes a constant unit vector specifying the direction of gravity.

    \item For $m \in \N$ and $f_i \in \mathcal P(\R^m)$ for $i=1,2$, we denote by  by $\m W_2(f_1,f_2)$ their $2$-Wasserstein distance and by $\Gamma(f_1,f_2) \in \m P(\R^m \times \R^m)$ the set of all transport plans between $f_1$ and $f_2$, see \cite{Santambrogio15} for details.

\item For a measure $f\in \m P(\R^{m_1})$ and a measurable map $X:\R^{m_1}\to \R^{m_2}$ we denote the pushforward measure $X\# f$ by $(X\# f)(A)=f(X^{-1}(A))$ for all measurable sets $A\subset \R^{m_2}$.

\item For a ball $B \subset \R^3$, we denote $\delta_{B} := \frac {\1_{B}}{| B|}$.

    \item We write $\Phi$ for the Oseen tensor, the fundamental solution of the Stokes equation, i.e.
    \begin{align}  \label{Oseen}
        \Phi(x) = \frac 1 {8 \pi} \left( \frac 1 {|x|} + \frac{x \otimes x} {|x|^3} \right).
    \end{align}
   We will use repeatedly that $\abs{\Phi(x)}+\abs{x}\abs{\nabla \Phi(x)}+\abs{x^2}\abs{\nabla^2\Phi(x)}\le \abs{x}^{-1}$ and that $\dv \Phi=0$. Moreover we use the convention that $\Phi(0)=0$.
    
    \item We denote $\dot H^1(\R^3)=\{w\in L^6(\R^3):\nabla w\in L^2(\R^3)\}$. This is a Hilbert space when endowed with the $L^2$ scalar product of the gradients.
    We denote its dual by $H^{-1}(\R^3)$.
    For $\rho \in \mathcal P(\R^3)$, we denote $L_\rho^p(\R^3)$ the Lebesgue spaces with respect to the measure $\rho$.
    We slightly abuse notation by allowing context to make clear what the target space is, i.e. whether the functions are scalar or vector valued. We will explicitly write the target space when we deem it necessary for clarity.

       \item For $1\le p\le \infty$ we write $\norm{\cdot}_p=\norm{\cdot}_{L^p(\R^3)}$. In integrals we only specify the domain of integration if it is not the whole space. By $p'$ we denote the H\"older dual of $p$, i.e. $1/p+1/p'=1$.
       
    \item For $h \in \mathcal P(\R^3 \times \R^3)$, we denote the moments
    \begin{align} \label{def:moments}
    \begin{aligned}
        m_k[h] :=   \int  |v|^k h \dd v, \\
        M_k[h] :=   \int  |v|^k h \dd v \dd x.
    \end{aligned}
    \end{align}
    Moreover, we denote
    \begin{align}\label{def:rho.j}
    \rho[h] := m_0[h], && j[h] := \int  v h \dd v,
    \end{align}
    so that $|j|\le m_1$. Finally, if $M_2[h] < \infty$, we define $\mV[h]:= j[h]/\rho[h] \in L^2_{\rho[h]}(\R^3)$, which should be understood as
\begin{align} \label{def:mV}
    \int \mV[h] \varphi \dd \rho[h] = \int v \varphi(x) \dd h(x,v) \qquad \text{for all }\varphi \in L^2_{\rho[h]}(\R^3).
\end{align}
This is well-defined by the Cauchy-Schwarz inequality.

\item For abbreviation, we denote $\St^{-1} \colon H^{-1}(\R^3) \to \dot H^1(\R^3)$ the solution operator for the Stokes equation, i.e. $j\mapsto \St^{-1}[j]$, where $u = \St^{-1} [j]$ is the unique weak solution  to 
\begin{align}\label{def:St}
    -\Delta u + \nabla p = j, \quad \dv u = 0 ~ \text{in } \R^3.
\end{align}
Note that for right-hand sides $j\in L^p(\R^3)$ with $1<p<\infty$, $u=\Phi\ast j$, and we will often use the estimate $\|\nabla^2 u\|_p\ls \|j\|_p$ which follows from the fact that $\nabla^2\Phi$ is a Calderon-Zygmund kernel.

Similarly, for  $\rho \in \mathcal P(\R^3)\cap L^{3/2}(\R^3)$,\footnote{The integrability assumption $L^{3/2}(\R^3)$ is convenient in order to test the equation with $\dot H^1(\R^3)$ functions but could be relaxed to define the operator $\Br^{-1}$.} we denote by  $\Br^{-1}_\rho: L^2_\rho(\R^3) \to \dot H^{1}(\R^3)$ the solution operator for the Brinkman equation (cf. Lemma~\ref{lem:u.Lipschitz} below), i.e.  $\m V \mapsto \Br^{-1}_\rho[\m V]$, where $u = \Br^{-1}_\rho[\m V]$ is the unique weak solution  to 
\begin{align} \label{def:Br}
    -\Delta u + \rho (u - \m V) + \nabla p = 0, \quad \dv u = 0 ~ \text{in } \R^3.
\end{align}

    \item For a fluid velocity $w$ that satisfies the Stokes equation in some domain, we write
    \begin{align}
        \sigma[w] = 2 e w- p \Id
    \end{align}
    for the fluid stress tensor where $ew = \frac 1 2 \left(\nabla w + (\nabla w)^T\right)$ is the strain and $p$ is the pressure associated to $w$ that we view as a Lagrange multiplier for the constraint $\dv w = 0$. For this reason we will also abusively use the same symbol $p$ for different pressures that can be recovered from corresponding fluid velocity fields.

     \item For $ X\in  \bra{\R^3}^N$ we write 
    \begin{align}
        \dmin \coloneqq \min_{i \neq j} |X_i - X_j|, \\
        d_{ij} \coloneqq \begin{cases}
                |X_i - X_j|, & \quad  \text{for} ~ i \neq j, \\
                \dmin & \quad  \text{for} ~ i = j.
                \end{cases}
    \end{align}
    Furthermore, for $\beta \in \R$ and $\dmin>0$, we denote
    \begin{align} \label{def.S}
        S_\beta \coloneqq \sup_i \sum_{j=1}^N \frac{1}{d_{ij}^\beta}.
    \end{align}
    We will usually write $\sum_j$ in short for $\sum_{j=1}^N$. The quantities $\dmin$, $d_{ij}$ and $S_\beta$ refer to the particle configuration $X$ given through the dynamics \eqref{eq:acceleration} --\eqref{eq:fluid.micro}. In particular, the quantities $\dmin$, $d_{ij}$ and $S_k$
    are time dependent.
    
    \item For given $N\in \N$ and for times $t \geq 0$, we write $F \in (\R^3)^N$ for the forces exerted by the particles on the fluid, i.e. 
    \begin{align} \label{def:F_i}
        F_i = -\int_{\partial B_i} \sigma[u_N] n \dd \mathcal H^2,
    \end{align}
   where $u_N$ is the solution to \eqref{eq:fluid.micro} corresponding to the specific particle configuration $X$. The sign is due to the orientation of the normal, which is the \emph{inner} normal of the fluid domain.

\item For $p\in [1,\infty)$ and a vector $W\in \bra{\R^3}^N$ we write
\begin{align} \label{euclidean}
    \abs{W}_p^p=\sum_{i=1}^N \abs{W_i}^p
\end{align}
and
\begin{align} \label{maximum}
   \abs{W}_\infty=\max_{i=1,\dots,N}\abs{W_i}.
\end{align}
\end{itemize}

%% file: 3_Brinkman.tex
\section{Some preliminary estimates on the Brinkman equation}
\label{sec:preliminary}
 
In this section, we gather some estimates on the Brinkman equation, the fluid equation in \eqref{eq:Vlasov.Stokes.gamma=1}.
It will be convenient to write the Brinkman equation in two different ways.
One the one hand as
\begin{align} \label{eq:Stokes_j}
- \Delta u + \nabla p = j - \rho u, \quad \dv u = 0 ~ \text{in } \R^3.
\end{align}
 And on the other hand as
 \begin{align} \label{eq:Stokes_mV}
 -\Delta u + \rho u +\nabla p = \rho \m V, \quad \dv u = 0 ~ \text{in } \R^3,
 \end{align}
 which is precisely $u = \Br^{-1}_\rho[\m V]$ (cf. \eqref{def:Br}).
Note that these equations  are obtained from the  
 formulation in \eqref{eq:Vlasov.Stokes.gamma=1} upon setting
$\rho = \rho[f], j = j[f], \mV = \mV[f]$ defined in \eqref{def:rho.j}--\eqref{def:mV}.
\begin{lem} \label{lem:u.Lipschitz}
Let $\varrho \in L^{3/2}(\R^3)$.
\begin{enumerate}[(i)]
\item  For every $\m V\in L^2_\varrho(\R^3;\R^3)$ there is a unique weak solution $u=\Br_\rho^{-1}[\m V]\in L^2_\rho(\R^3;\R^3)\cap \dot H^1(\R^3)$ to \eqref{eq:Stokes_mV}. The operator $\Br^{-1}_{\rho}$ is bounded and linear and it holds that
\begin{align} \label{est:A_rho}
   \| \Br^{-1}_{\rho}[\m V]\|_{L^2_\varrho}^2 +  \|\nabla \Br^{-1}_{\rho}[\m V]\|_2^2  \leq  \| \m V\|_{L^2_\varrho}^2.
\end{align}
\item Let additionally $j\in L^{6/5}(\R^3)$. Then there exists a unique weak solution $u \in \dot H^1(\R^3)$ to 
\eqref{eq:Stokes_j} and
    \begin{align}\label{eq:Brinkman_est1}
        \|\nabla u\|_2 + \|u\|_{L^2_\rho}  &\lesssim \|j\|_{6/5}.
    \end{align}
    Moreover, for $p \in [1,6]$ and $\frac 1 p = \frac 1 q + \frac 1 6$, it holds that
    \begin{align}\label{eq:Brinkman_est2}
        \|\nabla^2 u\|_p & \lesssim  \norm{j}_p+\norm{j}_{6/5}\norm{\rho}_q.
    \end{align}
    Furthermore, for all $r \in [6,\infty)$ and $s > 3$, $u$ satisfies the estimates
    \begin{align}\label{eq:Brinkman_est3}
        \| u\|_r &\lesssim \norm{j}_{3/2}+\norm{j}_{6/5}\bra{1+\norm{\rho}_2},\\
        \|\nabla^2 u\|_s+\|u\|_{W^{1,\infty}} &\lesssim  \norm{j}_{L^{6/5}\cap L^s}\bra{1+\norm{\rho}^2_{L^3\cap L^s}}.\label{eq:Brinkman_est4}
    \end{align}
\end{enumerate}

\end{lem}
\begin{proof}
Item (i) as well as well-posedness in (ii) and estimate \eqref{eq:Brinkman_est1}
    follow from H\"older and Sobolev inequality and a standard application of the Lax-Milgram theorem. 
    Estimate \eqref{eq:Brinkman_est2} follows from using $L^p$-regularity for the Stokes equation \eqref{eq:Stokes_j} with right-hand side $j-\rho u$ and upon using \eqref{eq:Brinkman_est1} and Sobolev embedding in
    \begin{align}
        \|\rho u \|_{p} \leq \|u\|_{6} \|\rho\|_{q}.
    \end{align}
    For \eqref{eq:Brinkman_est3} we use \eqref{eq:Brinkman_est2} with $p = 3/2$, $q=2$. By Sobolev embedding this implies a bound for $\norm{\nabla u}_{3}$. Interpolation with
    \eqref{eq:Brinkman_est1} and Sobolev embedding yields \eqref{eq:Brinkman_est3}.

    Finally, for \eqref{eq:Brinkman_est4}, we apply \eqref{eq:Brinkman_est2} with $p=2, q=3$ and obtain $\nabla u \in H^1$ which implies $u\in L^\infty(\R^3)$.
    Thus,
     \begin{align}
        \|\rho u \|_{s} \leq \|u\|_{\infty} \|\rho\|_{s},
    \end{align}
    and therefore, by using $L^s$ regularity of the Stokes-equation, we obtain $u \in \dot W^{2,s}(\R^3) \cap \dot H^1(\R^3)$ which embeds into $W^{1,\infty}(\R^3)$.
\end{proof}

A crucial ingredient for the proof of Theorem \ref{thm:hydrodynamic.limit} and Theorem \ref{thm:vlasov} is the following energy-energy-dissipation relation,
which is a straightforward adaptation from \cite{Hofer18InertialessLimit}.
\begin{lem}\label{lem:coercive} 
There is a universal constant $c_A > 0$ such that for all $\varrho \in L^1(\R^3) \cap L^{3/2}(\R^3)$ and all $\m V \in L^2_\varrho(\R^3)$, the inequality
\begin{align} \label{eq:coercive.0}
       \int \m V \cdot \left( \mV  - \Br^{-1}_{\rho}[\m V]\right) \dd \varrho \geq c_A \min\{1, \|\varrho\|_{3/2}^{-1} \}  \|\m V\|^2_{L^2_\varrho}
\end{align}
holds. In particular, for $f\in \m P(\R^3\times \R^3)$ with $M_2[f]<\infty$ and $\rho = \rho[f], \m V = \m V[f]$ defined as in \eqref{def:rho.j} and \eqref{def:mV}, and for any $w\in L^2_{\rho}(\R^3)$, it holds that
\begin{align} \label{coercive.full} 
       \int (v - w) \cdot \left( v  - w - \Br^{-1}_{\rho}[\m V - w]\right) \dd f(x,v) \geq c_A \min\{1, \|\rho\|_{3/2}^{-1} \}  \int |v-w|^2 \dd f(x,v). \qquad
\end{align}
\end{lem}
\begin{proof}
    Testing the Brinkman equation \eqref{eq:Stokes_mV} by the solution $u$ yields
    \begin{align}
        \|\nabla \Br^{-1}_{\rho}[\m V]\|^2_{2} = \int (\m V - \Br^{-1}_{\rho}[\m V]) \cdot  \Br^{-1}_{\rho}[\m V] \dd \varrho . 
    \end{align}
    From this, we obtain
    \begin{align} \label{dissipation}
        \int \m V \cdot \left( \mV  - \Br^{-1}_{\rho}[\m V]\right) \dd \varrho =  \|\nabla \Br^{-1}_{\rho}[\m V]\|^2_{2} + \|\m V - \Br^{-1}_{\rho}[\m V]\|^2_{L^2_\varrho}.
    \end{align}
    The assertion \eqref{eq:coercive.0} follows from 
    Lemma \ref{lem:frictionCoercive} below. We now turn to the proof of \eqref{coercive.full} and thus $\mV=\mV[f]$. We note that 
    \begin{align}\label{eq:V_orth}
        \int (v-\mV(x))\varphi(x)\dd f(x,v)=0,
    \end{align}
    for any function $\varphi\in L^2_\rho(\R^3)$ which directly follows from \eqref{def:mV}. Applying \eqref{eq:coercive.0} for $\m V-w$ and using \eqref{eq:V_orth} with $\varphi=\mV-w-\Br^{-1}_{\rho}[\m V-w]$ yields
    \begin{align}  
       \int (v - w) \cdot \left( \mV  - w - \Br^{-1}_{\rho}[\m V - w]\right) \dd f(x,v) \geq c_A \min\{1, \|\rho\|_{3/2}^{-1} \}  \int |\mV-w|^2 \dd f(x,v). \qquad
\end{align}
    Adding $\int (v-w)(v-p\mV)\dd f(x,v)$ on both sides produces the left-hand side of \eqref{coercive.full}. On the other hand, using \eqref{eq:V_orth} multiple times, we have
    \begin{align}
        \int (v-w)(v-\mV)\dd f(x,v)= \int v(v-\mV)\dd f(x,v)= \int \abs{v-\mV}^2\dd f(x,v).
    \end{align}
    Noting that 
    \begin{align}\label{eq:pythagoras}
        \int \abs{v-w}^2 \dd f(x,v)=\int \abs{v-\mV}^2+\abs{\mV-w}^2 \dd f(x,v),
    \end{align}
    where we again used \eqref{eq:V_orth}, finishes the proof. 
\end{proof}
\begin{lem}\cite[Lemma 3.1]{Hofer18InertialessLimit}
	\label{lem:frictionCoercive}	
	There exists a constant $c_0$, such that for all nonnegative functions $\sigma \in L^{3/2}(\R^3)$, all $h \in L^2_\sigma(\R^3)$, and all $w \in \dot H^1(\R^3)$, it holds that
	\[
		\| \nabla w \|_{2}^2 + \| w - h \|_{L^2_\sigma}^2 \geq c_0 \left( 1 + \|\sigma\|_{3/2}\right)^{-1} \|h\|_{L^2_\sigma}^2.
	\]
\end{lem}

The last two lemmas of this section are concerned with estimates for the difference of the solutions to two Stokes (respectively Brinkman) equations  associated with two densities in terms of the $2$-Wasserstein distance of these densities. These estimates will be crucial in the proof of all the three main theorems.

\begin{lem}\label{lem.diff.u.W_2.St}
       For $i =1,2,3$, let $\sigma_i \in \mathcal P(\R^3 \times \R^3) \cap L^p(\R^3)$ for some $p > 3$.
     For $i=1,2$, let $u_i\in \dot H^1(\R^3)$ be the unique weak solutions to 
     \begin{align}
    -\Delta u_i + \nabla p = \sigma_i g  , \qquad \dv u_i = 0.
     \end{align}
 Then,  
     \begin{align}
         \|u_1 - u_2\|_{L^2_{\sigma_3}}  \leq C \max_{ 1 \leq i \leq 3}\|\sigma_i\|_p^{\frac 1 {3 p'}} \W_2(\sigma_1,\sigma_2)  \label{u_h.W_2.St} 
     \end{align}
     where $C$ depends only on $p$ and where $p'$ is the H\"older dual of $p$.
\end{lem}
\begin{proof}
    We write $u_i = (\Phi g) \ast \sigma_i$. Then, for any $\gamma \in \Gamma(\sigma_1,\sigma_2)$ and $x \in \R^3$, we have
    \begin{align}
        (u_1 - u_2)(x) = \int (\Phi(x-y) - \Phi(x-z)) g \dd \gamma(y,z).
    \end{align}
    Using the decay properties of $\Phi$, we estimate
\begin{align}
    & |\Phi(x-z) - \Phi(x-y)| \\
    &\qquad \qquad \lesssim |z-y| \left(1 + \frac{1}{|x-y|^2} + \frac{1}{|x-z|^2}  \right)   \\
    &\qquad \qquad \lesssim  \left(|z-y| \left( 1 + \frac{1}{|x-y|} + \frac{1}{|x-z|}  \right)\right) 
     \left(1 + \frac{1}{|x-y|} + \frac{1}{|x-z|}  \right).
\end{align}
Using this and the Cauchy-Schwarz inequality, we arrive at
\begin{align} \label{Phi.convolution}
\begin{aligned}
    &|(u_1 - u_2)(x)|^2 \\
    & \quad \lesssim  \int |z-y|^2 \left(1 + \frac{1}{|x-y|^{2}} + \frac{1}{|x-z|^{2}}  \right)  \dd  \gamma(y,z)   \int 1 +  \frac{1}{|x-y|^{2}} + \frac{1}{|x-z|^2}   \dd  \gamma(y,z). \qquad
\end{aligned}
\end{align}
Observe that for any $\mu \in \mP(\R^3) \cap L^{p}(\R^3)$ with $p > 3$, we have
\begin{align} \label{frac.integrable.3}
    \sup_{x \in \R^3} \int  \frac{1}{|x-y|^{2}} \dd \mu(y) \leq C \|\mu\|_{p}^{\frac {2}{3 p'}},
\end{align}
where the constant $C$ depends only on $p$.
Indeed, it suffices to consider $x = 0$ and then, for $R>0$, we split the integral and apply H\"older's inequality to obtain
\begin{align}
    \int  |y|^{-2} \dd \mu(y) \leq \left(\int_{|y| \leq R} |y|^{-{2p'}} \right)^{\frac 1 {p'}} \|\mu\|_p + R^{-2} \|\mu\|_1.
\end{align}
For $p > 3$, the first term is bounded by $C R^{\frac{3}{p'} -2} \|\mu\|_p$. Optimizing in $R$ yields \eqref{frac.integrable.3}.

Thus, using that $\sigma_i \in L^p(\R^3)$ with $p > 3$, we obtain that the second integral on the right-hand side of \eqref{Phi.convolution} is uniformly bounded in $x$ by $\|\sigma_1\|_{p}^{2/(3 p')} + \|\sigma_2\|_{p}^{2/(3 p')}$.
Since $\sigma_3 \in L^p(\R^3)$ with $p > 3$, too, integrating \eqref{Phi.convolution}, using Fubini and again \eqref{frac.integrable.3} yields
\begin{align}
    \|u_1 - u_2\|^2_{L^2_{\sigma_3}}  \leq C \max_{ 1 \leq i \leq 3}\|\sigma_i\|_p^{\frac 2 {3 p'}} \int |y - z|^2 \dd \gamma(y,z).
\end{align}
Choosing an optimal transport plan $\gamma$ finishes the proof.
\end{proof}

\begin{rem}
    One can show the more general and stronger estimate  
    \begin{align}
        \|\sigma_1 - \sigma_2\|^2_{W^{-1,r}} \lesssim \max\{\|\sigma_1\|^{\frac 1 {p'}}_q,\|\sigma_2\|^{\frac 1 {p'}}_q \}  \W_p(\sigma_1,\sigma_2)
    \end{align}
for all $ 1 \leq p,q,r \leq \infty$ which satisfy
\begin{align}
    \frac 1 {r'} + \frac 1 p + \frac 1 {q p'} = 1.
\end{align}
This can be proved by exploiting properties of geodesics in Wasserstein spaces (see e.g. \cite[Proposition 5.1]{Hofer&Schubert} for the special case $q = \infty$).
\end{rem}

 \begin{lem} \label{lem.diff.u.W_2}
     For $i =1,2$, let $h_i \in \mathcal P(\R^3 \times \R^3) \cap L^\infty(\R^3 \times \R^3)$ with $M_9[h_2] < \infty$  and $\rho[h_i] \in L^p(\R^3)$, $i=1,2$  for some $p > 4$, and let $\sigma \in \mathcal P(\R^3) \cap L^p(\R^3)$.
     Let $u_i \in \dot H^1(\R^3)$ be the unique weak solutions to 
     \begin{align}
    -\Delta u_i + \nabla p = j[h_i] - \rho[h_i]u_i  , \qquad \dv u_i = 0.
     \end{align}
 Then,  
     \begin{align}
         \|u_1 - u_2\|_{L^2_\sigma}  \leq C \W_2(h_1,h_2)  \label{u_h.W_2} 
     \end{align}
     where $C$ depends only on $K,p$,  $\|\sigma\|_{p}$,   $\|\rho[h_i]\|_{q}$ for $i=1,2$, $M_9[h_2]$, and $\|u_2\|_{W^{1,\infty}(\R^3)}$.
\end{lem}
\begin{proof}
\noindent \textbf{Step 1:} \emph{Splitting of $u_1-u_2$.}
 Let $\gamma \in \Gamma(h_1,h_2)$ be an optimal transport plan. 
 We define $w$ as the solution to
   \begin{align}
        -\Delta w + \nabla p = \int (v_1 - v_2 - w) \gamma(\cdot,\dd v_1, \dd x_2, \dd v_2), \quad \dv w = 0,
    \end{align}
    and $r_1\coloneqq u_1-w$ which is the solution to
    \begin{align}
        -\Delta r_1 +  \nabla p = \int (v_2 - r_1) \gamma(\cdot,\dd v_1, \dd x_2, \dd v_2), \quad \dv r_1 = 0.
    \end{align}
    Moreover, we introduce $r_2$ as the solution to
\begin{align}
    -\Delta r_2 + \nabla p = \int (v_2 - u_2) \gamma(\cdot,\dd v_1, \dd x_2, \dd v_2), \quad \dv r_2 = 0.
\end{align}
Then,  $u_1-u_2=w+(r_1 - r_2) + (r_2-u_2)$.

\medskip

\noindent \textbf{Step 2:} \emph{Estimate of $w$.}   
    Denoting $\rho_i \coloneqq \rho[h_i]$, and testing the equation for $w$ by itself yields 
    \begin{align}
        \|\nabla w \|^2_{2} + \|w\|^2_{L^2_{\rho_1}} = \int w(x_1) \cdot (v_1 - v_2) \dd \gamma(x_1, v_1, x_2, v_2),
    \end{align}
    and hence by using $ab\le \tfrac 12(a^2+b^2)$ and absorbing one term on the left-hand side, 
    \begin{align} \label{est:Step2}  
        \|\nabla w \|_{2}^2 + \|w\|_{L^2_{\rho_1}}^2 \leq \int |v_1 - v_2|^2 \dd \gamma(x_1, v_1, x_2, v_2) \lesssim W^2_2(h_1,h_2).
    \end{align}

\medskip
    
\noindent \textbf{Step 3:}  \emph{Proof that
\begin{align} \label{est:Step3}
   \|r_2 - u_2\|^2_{L^2_\sigma} \lesssim  \int |y-z|^2  \dd  \gamma(y,v_1,z,v_2).
\end{align}
}   
The proof is similar to the proof of Lemma \ref{lem.diff.u.W_2.St}.
 We observe that
\begin{align}\label{eq:r2u2}
    (r_2 - u_2)(x) &= \int \Phi(x-z)(u_2(z) - v_2) - \Phi(x-y)(u_2(y) - v_2) \dd \gamma(y,v_1,z,v_2). 
\end{align}
Using the decay properties of $\Phi$, we estimate
\begin{align}\label{eq:Phi_diff}
\begin{aligned}
     |\Phi(&x-z)(u_2(z) - v_2) - \Phi(x-y)(u_2(y) - v_2)| \\
    & \lesssim |z-y| \left(1 + \frac{1}{|x-y|^2} + \frac{1}{|x-z|^2}  \right)  (|v_2| +\|u_2\|_{W^{1,\infty}(\R^3)} ) \\
    & \lesssim  \left(|z-y| \left( 1 + \frac{1}{|x-y|^{\frac 9 8}} + \frac{1}{|x-z|^{\frac 9 8}}  \right)\right) 
     \left(1 + \frac{1}{|x-y|^{\frac 7 8}} + \frac{1}{|x-z|^{\frac 7 8}}  \right) (1 + |v_2|).
     \end{aligned}
\end{align}
Inserting \eqref{eq:Phi_diff} in \eqref{eq:r2u2} and applying H\"older's inequality with $\frac 1 2 + \frac 7 {18} + \frac 1 9 = 1$ yields
\begin{align} \label{Phi.v_2}
\begin{aligned}
    |(r_2 - u_2)(x)| 
     &\lesssim  \left( \int |z-y|^2 \left(1 + \frac{1}{|x-y|^{\frac 9 4}} + \frac{1}{|x-z|^{\frac 9 4}}  \right)  \dd  \gamma(y,v_1,z,v_2) \right)^{\frac 1 2}  \\
    & \times \left( \int 1 +  \frac{1}{|x-y|^{\frac 9 4}} + \frac{1}{|x-z|^{\frac 9 4}}   \dd  \gamma(y,v_1,z,v_2)\right)^{\frac 7 {18}}  \left( \int 1 +  |v_2|^9  \dd  \gamma(y,v_1,z,v_2) \right)^{\frac 1 {9}}.
\end{aligned}
\end{align}
Analogously to \eqref{frac.integrable.3} we have for $\mu \in \mP(\R^3) \cap L^{p}(\R^3)$ with $p > 4$ that
\begin{align} \label{frac.integrable}
    \sup_{x \in \R^3} \int  \frac{1}{|x-y|^{\frac 9 4}} \dd \mu(y) \leq C \|\mu\|_{p}^{\frac{3p'}{4}},
\end{align}
where the constant $C$ depends only on $p$, and $p'$ is the H\"older dual of $p$.

Thus, using that $\rho[h_i] \in L^p(\R^3)$ with $p > 4$, we obtain that the second last integral in \eqref{Phi.v_2} is uniformly bounded in $x$.
Moreover, as $M_9[h_2] < \infty$,  the last integral in \eqref{Phi.v_2} is bounded. Since $\sigma \in L^p(\R^3)$ with $p > 4$ by assumption, integrating \eqref{Phi.v_2}, using Fubini and again \eqref{frac.integrable} yields \eqref{est:Step3}.

\medskip 

\noindent \textbf{Step 4:} \emph{Estimate of $r_1 -r_2$.}
We observe
\begin{align}
    r_1 - r_2 = \Br^{-1}_{\rho_1}[u_2 - r_2],
\end{align}
where we recall the notation $\Br^{-1}$ from \eqref{def:Br}.
Hence, by \eqref{est:A_rho},
\begin{align} \label{est:Step4}  
   \| \nabla (r_1 - r_2)\|_{2} + \| r_1 - r_2\|_{L^2_{\rho_1}} &\lesssim \|u_2 - r_2\|_{L^2_{\rho_1}}\ls \W_2^2(h_1,h_2) ,
\end{align}
where we applied estimate \eqref{est:Step3} in the second inequality with $\sigma = \rho_1$.

\medskip 

\noindent \textbf{Step 5:} \emph{Conclusion.}
For $w$, and $r_1 -r_2$ we use that by H\"older and Sobolev inequalities we have for any $\psi \in \dot H^1(\R^3)$ that
$\|\psi \|_{L^2_\sigma} \leq \|\nabla \psi\|_{2} \|\sigma\|_{3/2}^{ 1/ 2} $.
Combining this with the estimates \eqref{est:Step2}, \eqref{est:Step3} and \eqref{est:Step4}  yields \eqref{u_h.W_2}.
\end{proof}

%% file: 4_Existence.tex
\section{Global well-posedness and stability}

This Section is devoted to the proof of Theorem \ref{thm:existence}.

\label{sec:well-posedness}

Before we start with the main proof, we need the following estimates relating the Lebesgue norms of the moments $m_k$ and the moments $M_k$ defined in \eqref{def:moments}. The estimates can also be found in \cite{Hamdache98}, but we include the short proof for the reader's convenience.

\begin{lem} \label{lem:moments}
    Let $h \in L^\infty(\R^3 \times \R^3)$.
    Then, for all $0 \leq l \leq k$ there holds
    \begin{align}
         \| m_l[h]\|_{\frac{3+k}{3+l}} \leq C   \|h\|_\infty^{\frac {k-l} {3+k}} M_k[h]^{\frac {3+l} {3+k}},
    \end{align}
    for some universal constant $C$.
\end{lem}
\begin{rem}\label{rem:moments}
    We recall that $\rho[h] = m_0[h]$ and $|j[h]| \leq m_1[h]$. In particular, if $h$ satisfies $h \in L^\infty(\R^3\times\R^3)$ and $M_K[h] < \infty$, then
    \begin{align}
        \|\rho[h]\|_{(3+K)/3} + \|j[h]\|_{(3+K)/4} \le C, 
    \end{align}
   where $C$ depends only on $\|h\|_\infty$ and $M_K[h]$. In particular, for $K>9$ (and assuming also $h \in \m P(\R^3\times \R^3)$), by Lemma~\ref{lem:u.Lipschitz}, the corresponding solution to the Brinkman equation \eqref{eq:Stokes_j} then satisfies $\|u\|_{W^{1,\infty}} \leq C$.
\end{rem}
\begin{proof} 
    Let $R > 0$. 
    Then
    \begin{align}
       m_l[h] = \int_{|v| \leq R} |v|^l h \dd v + \int_{|v| \geq R} |v|^l h \dd v \lesssim R^{3+l} \|h\|_\infty  + \frac 1 {R^{k-l}} m_k[h].
    \end{align}
    Choosing
    \begin{align}
        R = m_{k}[h]^{\frac 1 {3+k}} \|h\|_\infty^{-\frac 1 {3+k}}
    \end{align}
    yields 
\begin{align}
     m_l[h] \lesssim m_{k}[h]^{\frac {3+l} {3+k}} \|h\|_\infty^{\frac {k-l} {3+k}},
\end{align}
which implies the assertion. 
\end{proof}

\begin{proof}[Proof of Theorem \ref{thm:existence}]
In this proof, since $\lambda$ is fixed, we allow (implicit) constants to depend on $\lambda$.

\medskip

\noindent \textbf{Step 1:} \emph{Setup of the fixed-point operator.}
We argue by a contraction argument in the space
\begin{align}
    \mathcal X_T \coloneqq \left\{h \in C((0,T); w - \mathcal P(\R^3 \times \R^3) : \|h\|_{\infty} \leq \|f^0\|_\infty e^{3 \lambda T}, \sup_{t \in (0,T)} M_K(h(t)) \leq M_K(f^0)  + 1 \right\},
\end{align}
endowed with the topology induced by
\begin{align}
    d(h_1,h_2) = \sup_{t \in (0,T)} \W_2(h_1,h_2).
\end{align}
Note that $\mathcal X_T$ is a complete metric space.

We define the operator $S : \mathcal X_T \to \mathcal X_T$
by $h \mapsto f_h$, where $f_h$ solves
\begin{align}\label{f.h}
    \partial_t  f_h + v \cdot \nabla_x  f_h +  \lambda \dv_v((g+u_h - v) f_h) &= 0, \qquad f(0) = f^0,
\end{align}
and where $u$ solves
\begin{align} \label{u.h}
    -\Delta u_h + \nabla p = j[h] - \rho[h]u_h  , \qquad \dv u_h = 0.
\end{align}

\smallskip

\noindent\textbf{Substep 1.1:} \emph{Existence and uniqueness of solutions  $f_h$  to \eqref{f.h}.}
We need to show that $u_h$ is continuous in time and globally Lipschitz in the spatial variable. Then, the problem \eqref{f.h} is well-posed, and $f_h$ is given by the method of characteristics through the push-forward of $f^0$ under the flow of $u_h$. More precisely, defining $X_h,V_h$ by
\begin{align} \label{def:char}
\left\{\begin{array}{rlrr}
    \dot X_h(t;s,x,v) &= V_h(t;s,x,v), \qquad &X_h(s;s,x,v) &= x, \\
    \dot V_h(t;s,x,v) &=  \lambda(g + u_h(t,X_h(t;s,x,v)) - V_h(t;s,x,v)), \qquad &V_h(s;s,x,v) &= v,
    \end{array}\right.
\end{align}
we have 
\begin{align} \label{eq:f_pushf.0}
    f_h(t)=(X(t;0),V(t;0))\# f^0
\end{align}
in terms of the measures and 
\begin{align}\label{eq:f_pushf}
    f_h(t,x,v) = e^{3 \lambda t}f^0(X_h(0;t,x,v),V_h(0;t,x,v))
\end{align}
in terms of the densities.
Combining Remark~\ref{rem:moments} and Lemma~\ref{lem:u.Lipschitz}, we see that the function $u_h$ in \eqref{u.h} satisfies
\begin{align}
    \sup_{t\in(0,T)}\|u_h(t)\|_{W^{1,\infty}(\R^3)} \lesssim 1.
\end{align}
The fact that $u$ is also continuous in time follows from a local (to every ball $B_1(x)$, $x\in \R^3$) application of the Gagliardo-Nirenberg inequality and Lemma~\ref{lem.diff.u.W_2} with $\sigma = \frac{\1_{B_1(x)}}{|B_1(x)|}$, so that 
\begin{align} 
\begin{aligned}
   \|u_{h}(t) - u_{h}(s)\|_{L^\infty(B_1(x))} &\lesssim \|u_{h}(t) - u_{h}(s)\|_{L^2(B_1(x))}^{ 2/5} \|u_{h}(t) - u_{h}(s)\|^{3/5}_{W^{1,\infty}(B_1(x))} \\
   &\lesssim \W_2(h(t),h(s))^{2/5} .
   \end{aligned}\label{est:continuity.u}
\end{align}

\smallskip 
 
\noindent\textbf{Substep 1.2:} \emph{$f_h \in \mathcal X_T$.} 
By the previous substep, $f_h$ is given through \eqref{eq:f_pushf}. In particular, we deduce
\begin{align}
    \|f_h\|_{L^\infty((0,T) \times \R^3 \times \R^3)} \leq e^{3 \lambda T} \|f^0\|_\infty. \label{f.infty}
\end{align}
In order to show that the map $S$ is well defined, it remains to propagate the moment bounds and show that $f_h \in C(0,\infty;w-\mathcal P(R^3 \times \R^3))$.

We estimate for $k \geq 3$ using \eqref{eq:f_pushf.0}, Lemma \ref{lem:moments}, and Young's inequality,
\begin{align} \label{moment.time.derivative}
\begin{aligned}
    \dot M_k[f_h] &= \lambda k \int_{\R^3 \times \R^3} |v|^{k-2} v \cdot (g + u_h -v) \dd f_h (x,v)\\
    & \ls \lambda k (M_{k-1}[f_h] + \|u_h\|_{3+k}\|m_{k-1}[f_h]\|_{\frac{3+k}{2+k}} - M_k[f_h]) \\
    &\lesssim 1 + M_k[f_h] +   \|u_h\|_{3+k} M_k[f_h]^{\frac{2+k}{3+k}} 
    \\& \lesssim 1 + M_k[f_h] + \|u_h\|^{3+k}_{3+k},   
\end{aligned}
\end{align}
where time dependencies are implicitly understood.
Combining Lemma \ref{lem:moments} and Lemma \ref{lem:u.Lipschitz},  $\|u_h\|^{3+k}_{3+k}$ is bounded in terms of $M_K[h]$ and $\|h\|_{\infty}$. Hence, a Gronwall argument shows that, choosing $T > 0$ sufficiently small, $M_K[f_h] \leq M_K[f^0] + 1$.
Finally, to show that $f_h \in C(0,\infty;w-\mathcal P(\R^3 \times \R^3))$, for $0 \leq s \leq t < T$, we compute
\begin{align} \label{cont.rho}
\begin{aligned}
    \int |X_h(t;0,x,v) - X_h(s;0,x,v)|^2 \dd f^0(x,v) &\leq \int (t-s) \int_s^t |V_h(\tau;0,x,v)|^2 \dd \tau \dd f^0(x,v)  \\
    & = (t-s) \int_s^t M_2[f_h(\tau)] \dd \tau  \lesssim (t-s)^2,
    \end{aligned}
\end{align}
and 
\begin{align}
    &\int |V_h(t;0,x,v) - V_h(s;0,x,v)|^2 \dd f^0(x,v) \\
    &\qquad \qquad \leq \lambda \int (t-s) \int_s^t |g + u_h(\tau,X(\tau,0,x,v)) - V_h(\tau;0,x,v)|^2 \dd \tau \dd f^0(x,v) \\
    & \qquad \qquad  \lesssim (t-s)^2,
\end{align}
where we used that $u_h\in L^2_{\rho[h]}(\R^3)$ is bounded uniformly. Combining these two estimates and using $(X(t;0),V(t;0),X(s;0),V(s;0))\# f^0$ as a competitor transport plan yields $f_h \in C(0,\infty;w-\mathcal P(\R^3 \times \R^3))$.

\medskip

\noindent \textbf{Step 2:} \emph{Contraction property for small times.}
Consider $h_1,h_2 \in \mathcal X_T$ and corresponding solutions $f_i \coloneqq f_{h_i}$ with $u_i \coloneqq u_{h_i}$ the solutions to \eqref{u.h}. Let $X_i,V_i$ be the corresponding characteristics and
\begin{align}
    P(t)\coloneqq \frac 1 2 \int \left(|X_1(t;0,x,v) - X_2(t;0,x,v)|^2 + |V_1(t;0,x,v) - V_2(t;0,x,v)|^2\right)  \dd f^0(x,v)  .
\end{align}
Then
\begin{align}
    \dot P &= \int\left((X_1- X_2)(V_1 - V_2) +  \lambda (V_1 - V_2)(u_1 \circ X_1 - u_2 \circ X_2 - (V_1- V_2)\right) \dd f^0  \\
    & \lesssim P (1 + \|\nabla u_2\|_\infty) + \sqrt P \left(\int |(u_1 -u_2)\circ X_1|^2 \dd f^0 \right)^{\frac 1 2}.
\end{align}
Moreover, changing variables leads to
\begin{align}
    \int |(u_1 -u_2)\circ X_1|^2 \dd f^0 = \int |u_1 -u_2|^2 \dd f_1   = \|u_1 - u_2\|_{L^2_{\rho_1}}^2. 
\end{align}
Applying Lemmas \ref{lem:u.Lipschitz}, \ref{lem:moments} and \ref{lem.diff.u.W_2} to bound $\|\nabla u_2\|_\infty$ and $\|u_1 - u_2\|_{L^2_{\rho_1}}^2$ yields
\begin{align} \label{P.dot}
    \dot P   & \lesssim P  + \sqrt P \W_2(h_1,h_2) \lesssim P + \W^2_2(h_1,h_2).
\end{align}
Gronwall's inequality and $\mathcal W_2(f_1(t),f_2(t)) \lesssim \sqrt P$
then yields the desired contraction property for small times.

\medskip

\noindent \textbf{Step 3:} \emph{Global well-posedness.}
In order to obtain global existence and uniqueness, it suffices to show that the bounds involved in the definition of $\mathcal X_T$ cannot blow up for the solution $f$ for finite $T$.
Then the assertion follows by a standard continuation argument which we do not detail.

For $\|f(t)\|_\infty$, this is just \eqref{f.infty}. 
For estimating the moments, we first use the energy identity \eqref{energy.identity}, which follows from the identity in \eqref{moment.time.derivative} for $k=2$ and using the fluid equation tested with $u$.
\begin{align}
    \dot M_2[f_\lambda]&=2\lambda\int v\cdot(g+u_\lambda-v)\dd f_\lambda(x,v)\\
    &=-2\lambda\int |u_\lambda-v|^2\dd f(x,v)-2\lambda\|\nabla u_\lambda\|_2^2+2\lambda\int v\cdot g\dd f(x,v).
\end{align}
Rearranging the terms yields \eqref{energy.identity}, which we recall here 
\begin{align}\label{energy.identity2}
     \frac 1 2\dot  M_2[f] + \lambda \|\nabla u\|_2^2 + \lambda \int |u-v|^2 \dd f(x,v) = \lambda \int v \cdot g  \dd f(x,v) \leq \lambda M_2[f]^{\frac 1 2}.
\end{align}
This serves to propagate the second moment, as $M_2$ is controlled on $[0,T]$ (with at most quadratic growth in time). Note that through Lemma \ref{lem:moments} we then control $\|j\|_{5/4}$ and  thus $\|j\|_{6/5}$. Hence, by Lemma \ref{lem:u.Lipschitz} this  implies a bound on $\|u\|_6 \lesssim \|\nabla u\|_2$ on $[0,T]$. Then, using \eqref{moment.time.derivative} with $k = 3$ yields control of $M_3$. 
Next, by Lemma \ref{lem:moments}, control of $M_3$ yields control of $\|\rho\|_{2}$ and $\|j\|_{3/2}$. In turn, Lemma \ref{lem:u.Lipschitz} implies control of $\|u\|_r$ for all $r \in [6,\infty)$. Using again \eqref{moment.time.derivative} we therefore control all moments on $[0,T]$ that we control initially.

\medskip

\noindent \textbf{Step 4:} \emph{Stability estimate.}
Let $f_i$, $i=1,2$ be solutions with initial data $f_i^0$ satisfying the same assumptions as $f^0$. Let $\gamma_0\in \Gamma(f_1^0, f_2^0)$ and define
\begin{align}
    P(t) \coloneqq  \frac 1 2 \int \left(|X_1(t;0,x_1,v_1) - X_2(t;0,x_2,v_2)|^2 + |V_1(t;0,x_1,v_1) - V_2(t;0,x_2,v_2)|^2\right)\\
    \qquad\times \dd \gamma_0(x_1, v_1, x_2, v_2),
\end{align}
where $(X_i,V_i)$ are the characteristics associated with the two solutions.
A computation similar to the one in Step 2 implies $\dot P \lesssim P$. We conclude by applying Gronwall's inequality.
\end{proof}

%% file: 5_HydrodynamicLimit.tex
\section{Hydrodynamic limit}
\label{sec:hydrodynamic.limit}
A key step towards the  proof of the hydrodynamic limit, Theorem \ref{thm:hydrodynamic.limit},
is to obtain estimates on the solutions $(f_\lambda, u_\lambda)$ which are uniform in $\lambda$. Since the $L^\infty$ norm of $f_\lambda$ blows up as $\lambda \to \infty$ due to \eqref{f.infty}, most of the estimates from the previous section cannot directly be used to this end. 
A key ingredient for the proof is the following Lemma which is taken from \cite{Hofer18InertialessLimit}. It asserts that the concentration of the density responsible for this blowup of $f_\lambda$ only happens with respect to the velocity variable. This is crucial in order to obtain uniform estimates for $\rho_\lambda := \rho[f_\lambda]$.
For the convenience of the reader, we include the proof of the lemma.

\begin{lem}[{\cite[Lemma 3.4]{Hofer18InertialessLimit}}]
	\label{lem:biLipschitzV}
	Let $f^0\in \m P(\R^3\times \R^3)\cap L^\infty(\R^3\times \R^3)$ satisfy $M_K[f^0] < \infty$ for some $K > 9$ and let $(f_\lambda,u_\lambda)$ be the weak solution to \eqref{eq:Vlasov.Stokes} with corresponding characteristics $(X_\lambda, V_\lambda)$ (cf. \eqref{def:char}). Denote $\rho_\lambda=\rho[f_\lambda]$, let $T > 0$ and assume $\lambda \geq 4(1+ \|\nabla u_\lambda\|_{L^\infty((0,T) \times \R^3)})$. Then, for all $t < T$ and all $x \in \R^3$, the map
	\[
		v \mapsto V_\lambda(0;t,x,v)
	\]
	is bi-Lipschitz. In particular its inverse $W_\lambda(t,x,w)$ is well defined, and
	\begin{equation}
		\label{eq:rhoByW}
		\rho_\lambda(t,x) = \int_{\R^3} e^{ 3 \lambda t} f^0(X(0;t,x,W(t,x,w)),w) \det \nabla_w W_\lambda(t,x,w)  \dd w.
	\end{equation}
	Moreover, denoting
	\begin{equation}
		\label{eq:defM}
		A_\lambda(t) \coloneqq \exp \left(\int_0^t   2 \|\nabla u_\lambda(s,\cdot) \|_{\infty} \dd s \right),
	\end{equation}
	we have
	\begin{align}
		|\nabla_v V_\lambda(0,t,x,v)| &\leq A_\lambda(t) e^{\lambda t}, \label{eq:LipschitzV} \\
		|\nabla_w W_\lambda(t,x,w)| &\leq A_\lambda(t) e^{-\lambda t}, \label{eq:LipschitzW} \\
		\abs{\det \nabla_w W(t,x,w)} &\leq A_\lambda(t)^3 e^{-3\lambda t}. \label{eq:JacobianW}
	\end{align}
\end{lem}

\begin{proof}
	We fix $t$, $x$, $v_1$, and $v_2$ and define
\begin{align}
	a(s) &= | X_\lambda(s;t,x,v_1) - X_\lambda(s;t,x,v_2) |,\\
	b(s) &= | V_\lambda(s;t,x,v_1) - V_\lambda(s;t,x,v_2)|.
\end{align}
Then, 
\begin{align}
	|\dot{a}| & \leq b, \qquad &&a(t) = 0, \\
	\dot{b} &\leq \lambda ( \|\nabla u_\lambda(s,\cdot) \|_{\infty} a - b), \qquad &&b(t) = |v_1 - v_2|.
\end{align}
With $\alpha(s) \coloneqq  \|\nabla u_\lambda(s,\cdot) \|_{\infty}$ and $\beta = 0$, we can apply 
Lemma \ref{lem:ODEEstimates}\ref{it:ODEEstimates1} to deduce
\[
	b(t) \leq b(0) A_\lambda(t) e^{-\lambda t},
\]
which implies
\begin{equation}
	\label{eq:LipschitzInverse}
	b(0) \geq A_\lambda(t)^{-1} e^{\lambda t} |v_1 - v_2|.
\end{equation}
Note that the first inequality in  \eqref{eq:ODEEstimates1a} also implies 
\[	
	a(t) \leq \frac{2}{\lambda} b(t).
\]
Hence,
\[
	\dot{b} \geq\lambda ( - \|\nabla u_\lambda(s,\cdot) \|_{\infty} a - b) \geq \left( - \lambda   - 2\|\nabla u_\lambda(s,\cdot) \|_{\infty} \right) b,
\]
and thus
\begin{equation}
	\label{eq:Lipschitz}
	b(0) \leq e^{\lambda t} A_\lambda(t) |v_1 - v_2|.
\end{equation}

Estimates \eqref{eq:LipschitzInverse} and \eqref{eq:Lipschitz} imply that the map $v \mapsto V_\lambda(0;t,x,v)$ is bi-Lipschitz
and yield the bounds \eqref{eq:LipschitzV}, \eqref{eq:LipschitzW}, and \eqref{eq:JacobianW}. 
\end{proof}

As we will need this several times, we show here how to get uniform control on $\rho_\lambda=\rho[f_\lambda]$ and $j_\lambda=j[f_\lambda]$ under the condition that the Lipschitz norm of $u_\lambda$ is already controlled uniformly.
\begin{lem}\label{lem:moment_control}
     Let $C_\ast<\infty$, $T>0$, $K > 9 $, and 
    let $f^0 \in \mP(\R^3 \times \R^3) \cap L^\infty(\R^3 \times \R^3)$ with $M_K[f^0] < \infty$. There exists a constant $C<\infty$ only depending on $T$, $K$, $M_K[f^0]$, $\|f^0\|_\infty$ and $C_\ast$ such that the following holds. Let $(f_\lambda,u_\lambda)$ be the unique solution 
    to \eqref{eq:Vlasov.Stokes} from Theorem \ref{thm:existence}. 
    Set 
    \begin{align} 
         T_\lambda\coloneqq \sup\{t \in [0,T] : \sup_{s \leq t} \|u_\lambda(s)\|_{W^{1,\infty}} \leq C_\ast  \}. 
    \end{align}
    Then, for all $t<T_\lambda$ it holds that  
    \begin{align} \label{moment.est.refined}
         \| m_l[f_\lambda(t)]\|_{\frac{3+K}{3+l}} \le C (A_\lambda(t)^3 \|f^0\|_\infty)^{\frac {K-l} {3+K}} M_K[f_\lambda(t)]^{\frac {3+l} {3+K}},
    \end{align}
    where $A_\lambda$ is as in Lemma~\ref{lem:biLipschitzV}. In particular
    \begin{align} \label{rho.j.T_lambda}
    \sup_{t < T_\lambda} M_K(f_\lambda(t))+\| \rho_\lambda(t)\|_{\frac{3+K}3} +  \| j_\lambda(t)\|_{\frac{3+K}4} \le C.    \end{align}
\end{lem}
\begin{proof}
    Throughout the proof, we use the notation $\rho_\lambda = \rho[f_\lambda], j_\lambda = j[f_\lambda]$. In line with our convention we will write $\lesssim$ if $A,B \in \R$ satisfy  $A \leq C B$ for some $C$ that depends only on $T$, $K$, $M_K[f^0]$, $\|f^0\|_\infty$ and on $C_\ast$. 

    We start by proving uniform bounds on $M_K[f_\lambda]$ on $(0,T_\lambda)$: 
Indeed, refining the estimate in \eqref{moment.time.derivative} leads to
\begin{align}\label{eq:ddtMk}
\begin{aligned}
    \dot M_K[f_\lambda] &\leq - \lambda K M_K[f_\lambda] + \lambda K (|g| + \|u_\lambda\|_\infty) M_{K-1}[f_\lambda] \\
    &\leq \lambda K \left(- M_K [f_\lambda] + (1+C_\ast)(M_K [f_\lambda])^{\frac{K-1}K} \right) \\
    & \leq- \lambda  M_K [f_\lambda] + C \lambda K (1+C_\ast^K), 
    \end{aligned}
\end{align}
for some $C$ that only depends on $K$.
Thus, by a comparison principle for ODEs,
\begin{align} \label{moments.T_lambda}
    \sup_{t < T_\lambda} M_K[f_\lambda(t)] \lesssim 1,
\end{align}
which is the first part of \eqref{rho.j.T_lambda}. In particular all moments $M_k[f_\lambda(t)]$ for $k<K$ also satisfy a uniform bound.

The  change of variables from Lemma \ref{lem:biLipschitzV} now allows  us to obtain uniform estimates similarly as in Lemma \ref{lem:moments} despite the growth of $\|f\|_\infty$.
More precisely,
we have
\begin{align} \label{change.variables}
 \begin{aligned} 
 	m_l[f_\lambda(t)] &= \int |v|^l f_\lambda(t,x,v) \dd v \\
  & = \int |V_\lambda(0;t,x,v)|^l e^{ 3 \lambda t}  f^0(X_\lambda(0;t,x,v),V(0;t,x,v))  \dd v \\
 	&= \int |w|^l \abs{\det \nabla_w W_\lambda(t,x,w)} e^{ 3 \lambda t} f^0(X_\lambda(0;t,x,W_\lambda(t,x,w)),w)   \dd w \\
 	&\leq \int_{B_R(0)} |w|^l A_\lambda(t)^3 f^0(X_\lambda(0;t,x,W_\lambda(t,x,w)),w)  \dd w \\
 	&\quad+ R^{l-k}\int_{B_R(0)^c} |w|^k  \abs{\det \nabla_w W_\lambda(t,x,w)} e^{ 3 \lambda t} f^0(X_\lambda(0;t,x,W_\lambda(t,x,w)),w)   \dd w, \\
 	& \lesssim R^{3+l} A_\lambda(t)^3 \|f^0\|_\infty + R^{l-K}   m_K[f_\lambda(t)],
 \end{aligned}
 \end{align}
 where in the last inequality we reversed the two changes of variables for the second term.
Choosing 
    \begin{align}
        R = \left(\frac{m_K[f_\lambda(t)]}{A_\lambda(t)^3 \|f^0\|_\infty} \right)^{\frac 1 {3+K}}
    \end{align}
    yields 
\begin{align}
     m_l[f_\lambda(t)] \lesssim (m_{K}[f_\lambda(t)])^{\frac {3+l} {3+K}} (A_\lambda(t)^3 \|f^0\|_\infty)^{\frac {K-l} {3+K}},
\end{align}
and taking the $L^{\frac{3+K}{r+l}}$-norm delivers \eqref{moment.est.refined}. Since $A_\lambda(t) \leq e^{2C_\ast t}$ on $(0,T_\lambda)$ and we control $M_K$ due to \eqref{moments.T_lambda}, estimate \eqref{moment.est.refined} implies
\eqref{rho.j.T_lambda}.
\end{proof}

\begin{lem} \label{lem:short.times}
    Under the assumptions of Theorem \ref{thm:hydrodynamic.limit}, the following holds true.
     If Condition \ref{cond:L^1L^infty} or \ref{cond:f.small} from Theorem \ref{thm:hydrodynamic.limit} holds,  then 
     there exist $T_0 > 0$ and $L > 0$, depending only on  $K$,  $M_K[f^0]$, $\|f^0\|_{\infty}$,  and, in case of Condition \ref{cond:L^1L^infty}, on $\|f^0\|_{L^1_v(\R^3;L^\infty_x(\R^3))}$,  such that
        \begin{align} \label{T_0}
           \sup_{t \leq T_0} \|u_\lambda(t)\|_{W^{1,\infty}} \leq L.
        \end{align}
\end{lem}
\begin{proof} \textbf{Step 1:} \emph{Control of $M_2$ under Condition  \ref{cond:L^1L^infty} or \ref{cond:f.small} for short times.}

With $A_\lambda(t)$ as in \eqref{eq:defM}, we define
\begin{align}
    T_0 \coloneqq \sup\{t > 0 : A_\lambda(t) \leq 2 \},
\end{align}
and argue that under Condition \ref{cond:L^1L^infty} or Condition \ref{cond:f.small} we have
\begin{align}
   \sup_{t<T_0}M_2[f_\lambda]\lesssim 1. \label{M2.final}
\end{align}

\noindent\textbf{Condition \ref{cond:L^1L^infty}:} \emph{ $f^0 \in L^1_v(\R^3;L^\infty_x(\R^3))$.} In this case, we revisit \eqref{change.variables}
to obtain
\begin{align}
    \int f_\lambda(t,x,v) \dd v 
  & = \int e^{ 3 \lambda t} f^0(X_\lambda(0;t,x,v),V_\lambda(0;t,x,v))\dd v \\
 	&= \int |\det \nabla_w W_\lambda(t,x,w)| e^{ 3 \lambda t} f^0(X_\lambda(0;t,x,W_\lambda(t,x,w)),w)    \dd w \\
 	&\leq \int A_\lambda(t)^3 f^0(X_\lambda(0;t,x,W_\lambda(t,x,w)),w) \dd w \\
 	& \leq A_\lambda(t)^3 \|f^0\|_{L^1_v L^\infty_x}.
\end{align}
In particular, 
\begin{align}\label{eq:rhobound}
    \|\rho_\lambda(t)\|_{\infty} \leq A_\lambda(t)^3 \|f^0\|_{L^1_v L^\infty_x}.
\end{align}
We recall the energy identity \eqref{energy.identity2}. Using \eqref{eq:pythagoras} once for $w=u_\lambda$ and once for $w=0$, and combining with Lemma~\ref{lem:frictionCoercive} implies 
\begin{align} \label{M_2.improved}
\begin{aligned}
\dot M_2[f_\lambda] &\leq 2 \lambda (M_2[f_\lambda]^{1/2} - c_0(1 + \|\rho\|_{3/2})^{-1}M_2[f_\lambda] ) \\ &\leq 2 \lambda (M_2[f_\lambda]^{1/2} - c_0(1 + \|\rho\|_{\infty}^{1/3})^{-1}M_2[f_\lambda] ). 
\end{aligned}
\end{align}
Thus we have, using \eqref{eq:rhobound}, for all $t<T_0$ the estimate
\begin{align}\label{eq:M2_C1}
     M_2^{1 /2}(t) \leq  M_2^{1/2}(0) + \frac{1 + \|\rho\|_\infty^{1/3}}{c_0}\le  M_2^{1/2}(0) + \frac{1 +2 \|f^0\|^{1/3}_{L^1_v L^\infty_x}}{c_0} \lesssim 1.  
\end{align}

\medskip

\noindent\textbf{Condition \ref{cond:f.small}:} \emph{$\|f^0\|_{\infty} \leq \eps_0$.}
Combining \eqref{M_2.improved} and \eqref{moment.est.refined} in form of 
\begin{align}
    \|\rho_\lambda\|_{3/2} \leq\|\rho_\lambda\|^{1/6}_{1}  \|\rho_\lambda\|_{5/3}^{5/6} \lesssim  \left((A_\lambda(t)^3 \|f^0\|_\infty)^{2/5} M_2[f_\lambda]^{3/5} \right)^{5/6} = A_\lambda(t)\|f^0\|_\infty^{1/3} M_2[f_\lambda]^{1/2}
\end{align}
yields
\begin{align}
    \dot M_2[f_\lambda] &\leq 2  \lambda (M_2[f_\lambda]^{1/2} - c_0(1 + \|\rho_\lambda\|_{3/2})^{-1}M_2[f_\lambda] ) \\
    &\leq 2\lambda \left(M_2[f_\lambda]^{1/2} - c \frac{M_2[f_\lambda]}{1 +A_\lambda(t)\eps_0^{1/3} M_2[f_\lambda]^{1/2}} \right)
\end{align}
for some $c >0$ that depends only on  $K$,  $M_K[f^0]$, $\|f^0\|_\infty$  and $T$.
Choosing $\eps_0 ^{1/3} = c/4$, this implies for all $t<T_0$ the estimate
\begin{align}\label{eq:M2_C2}
    M_2[f_\lambda] \le \max \{M_2[f_\lambda](0), 4c^{-2} \} \lesssim 1.
\end{align}

\medskip

\noindent\textbf{Step 2:} \emph{Proof of \eqref{T_0} and lower bound on $T_0$.}

We  first argue that it suffices to show \eqref{T_0} for a suitable $L$.
Indeed by definition of $T_0$ and $A_\lambda(t)$ (see \eqref{T_0} and \eqref{eq:defM}) a continuity argument then implies $T_0 \geq \log 2/L$ and the assertion follows.

In the following, we only consider times $t \leq T_0$. We denote by $\theta > 0$ some universal exponent which might change from line to line. By \eqref{moment.est.refined} (replacing $K$ by $2$) and \eqref{M2.final}, we have
\begin{align}\label{eq:54}
 \|j_\lambda\|_{5/4} \lesssim \|f^0\|_\infty^{1/5} M_2[f_\lambda]^{4/5} \lesssim 1.  
\end{align}
Moreover, 
\begin{align}\label{eq:1}
    \|j_\lambda\|_{1} \le \int |v| \dd f_\lambda(x,v) \leq M_2[f_\lambda]^{\frac 1 2} \lesssim 1,
\end{align}
and interpolation between \eqref{eq:54} and \eqref{eq:1} yields
\begin{align}
    \|j_\lambda\|_{6/5} \lesssim \|f^0\|_\infty^{\beta} M_2[f_\lambda]^{\theta} \lesssim  1.
\end{align}
Then, by \eqref{eq:Brinkman_est1} and Sobolev embedding,
\begin{align} \label{u_6.short}
    \|u_\lambda\|_6 \lesssim \|\nabla u_\lambda\|_2 \lesssim 1.
\end{align}
Proceeding as in \eqref{moment.time.derivative} and \eqref{eq:ddtMk}, and applying \eqref{moment.est.refined} we have
\begin{align} \label{moment.time.derivative.refined}
\begin{aligned}
    \dot M_k[f_\lambda]  & \ls \lambda k (M_{k-1}[f_\lambda] + \|u_\lambda\|_{3+k}\|m_{k-1}[f_\lambda]\|_{\frac{3+k}{2+k}} - M_k[f_\lambda]) \\
    &\leq  \lambda k \left(C(1 + \|u_\lambda\|^{3+k}_{3+k}) - \frac 1 2 M_k[f_\lambda]\right),
    \end{aligned}
\end{align}
where $C\ls 1$. Applying \eqref{moment.time.derivative.refined} with $k=3$ and using \eqref{u_6.short} we deduce
\begin{align}
    M_3[f_\lambda] \lesssim 1,
\end{align}
which, by \eqref{moment.est.refined}, implies
\begin{align}
    \|\rho_\lambda\|_2 + \|j_\lambda\|_{3/2} \lesssim  1.
\end{align}
Therefore, by \eqref{eq:Brinkman_est3}, for all $q < \infty$ we have the control
\begin{align}
   \|u_\lambda\|_q \lesssim 1.
\end{align}
Inserting this once more into \eqref{moment.time.derivative.refined} yields
\begin{align}
    M_K[f_\lambda] \lesssim 1,
\end{align}
and thus, by \eqref{moment.est.refined} and \eqref{eq:Brinkman_est4}, the bound
\begin{align} 
    \|u_\lambda\|_{W^{1,\infty}} \ls 1 .
\end{align}
This concludes the proof.
\end{proof}

\begin{proof}[Proof of Theorem \ref{thm:hydrodynamic.limit}]
\textbf{Step 1: }\emph{Structure of the proof.}
Let $L$ be as in Lemma \ref{lem:short.times} and consider  the time
\begin{align} \label{def:T_lambda}
    T_\lambda&\coloneqq \sup\{t \in [0,T] : \sup_{s \leq t} \|u_\lambda(s)\|_{W^{1,\infty}} \leq  C_\ast\}, \quad \text{where}\\
   C_\ast &\coloneqq \|u_\lambda(0)\|_{W^{1,\infty}} + \sup_{s \in [0,T]} \|u_\ast(s)\|_{W^{1,\infty}} + L + 1.
\end{align}

Note that $\rho_\ast(t) \in \m P(\R^3) \cap L^{1 + K/3}(\R^3)$ with $\|\rho_\ast(t)\|_{1 + K/3} = \|\rho^0\|_{1 + K/3} \lesssim 1$ for all $t \geq 0$ since $\rho_\ast$ solves a transport equation with the divergence free transport velocity $u_\ast + g$. Hence,  $\nabla^2 u_\ast$ is uniformly bounded in $L^{6/5}(\R^3)\cap L^{1+K/3}(\R^3)$, and this implies
\begin{align}\label{eq:u_ast_Lip}
   \sup_{t \in [0,\infty)} \|u_\ast\|_{W^{1,\infty}} \lesssim 1.
\end{align}

Furthermore, $\|u_\lambda(0)\|_{W^{1,\infty}}$ is independent of $\lambda$ and controlled through $M_K[f^0]$ and $\|f^0\|_\infty$ due to \eqref{eq:Brinkman_est4} and Lemma \ref{lem:moments}.
Thus, 
\begin{align}
    C_\ast \lesssim 1.
\end{align}
Since $\|u_\lambda\|_{W^{1,\infty}}$ is continuous due to \eqref{est:continuity.u}, we obtain  $T_\lambda > 0$. 

Estimate \eqref{rho.j.T_lambda} from Lemma~\ref{lem:moment_control} yields \eqref{uniform.bounds.thm} on the time interval $(0,T_\lambda)$. In Step 2, we prove the estimates \eqref{W.rho.hydrodynamic} and \eqref{W.f.hydrodynamic} for $t< T_\lambda$. In Step 3, we use a buckling argument to show that $T_\lambda = T$ for $\lambda \gg 1$ under the hypothesis of the theorem.
Throughout the proof, we use the notation $\rho_\lambda = \rho[f_\lambda], j_\lambda = j[f_\lambda]$.

\medskip

\noindent\textbf{Step 2:}\emph{ Proof of \eqref{W.rho.hydrodynamic} and \eqref{W.f.hydrodynamic} for $t < T_\lambda$.}

 We define $w_\lambda\coloneqq\St^{-1}(\rho_\lambda g) \in \dot H^1(\R^3)$, i.e. 
\begin{align}
    - \Delta w_\lambda + \nabla p = \varrho_\lambda g, \quad \dv w_\lambda = 0.
\end{align}
Then, by Lemma \ref{lem.diff.u.W_2.St} and \eqref{rho.j.T_lambda}, we have
\begin{align} \label{est:w_lambda.u_ast}
    \|w_\lambda - u_\ast\|_{L^2_{\rho_\ast}}+\|w_\lambda - u_\ast\|_{L^2_{\rho_\lambda}}+\|w_\lambda - u_\ast\|_{L^2(B_1(x))} \lesssim  \m W_2(\rho_\lambda, \rho_\ast).
\end{align}
We recall the notation $\mV_\lambda \coloneqq \mV[f_\lambda]$ from \eqref{def:mV} and note that 
\begin{align} \label{rearr.u}
   u_\lambda = \Br^{-1}_{\varrho_\lambda}[\mV_\lambda] = \Br^{-1}_{\varrho_\lambda}[\mV_\lambda - g - w_\lambda] +  \Br^{-1}_{\varrho_\lambda}[g + w_\lambda].
\end{align}
Since 
\begin{align} \label{w.repr}
    w_\lambda = \Br^{-1}_{\varrho_\lambda}[g + w_\lambda],
\end{align}
we obtain
\begin{align} \label{split.u-u_ast}
    u_\lambda - u_\ast = \Br^{-1}_{\varrho_\lambda}[\mV_\lambda - g - w_\lambda] + (w_\lambda - u_\ast).
\end{align}
We consider the relative energy
\begin{align}
    S(t) \coloneqq \frac 1 2 \int |v - g -w_\lambda(t,x)|^2  f_\lambda(t, \dd x, \dd v) = \frac 1 2 \int |V_\lambda - g -w_\lambda \circ X_\lambda|^2 \dd f^0,
\end{align}
where we applied \eqref{eq:f_pushf.0} and used the short notation $(V_{\lambda},X_\lambda) = (V_\lambda,X_\lambda)(t;0,x,v)$.
Then
\begin{align}\label{eq:DR}
\begin{aligned}
    \frac{\dd }{\dd t}  S &= - \lambda \int (V_\lambda - g - u_\lambda\circ X_\lambda) \cdot (V_\lambda- g -w_\lambda \circ X_\lambda) \dd f^0 \\
    &\quad - \int (\partial_t w_\lambda\circ X + V_\lambda \cdot \nabla w_\lambda\circ X_\lambda) \cdot (V_\lambda- g -w_\lambda \circ X_\lambda) \dd f^0\\
    &\eqqcolon -\lambda D + R.
    \end{aligned}
\end{align}
On the one hand, by \eqref{rearr.u} and \eqref{w.repr}
\begin{align}\label{eq:D}
    D = \int (v - g - w_\lambda ) \cdot (v - g - w_\lambda - \Br^{-1}_{\rho_\lambda}[\mV_\lambda - g - w_\lambda])  f_\lambda(t, \dd x, \dd v)  \gtrsim  S,
\end{align}
where we used \eqref{coercive.full} and \eqref{rho.j.T_lambda} in the last estimate.
On the other hand, by arguing as for \eqref{eq:u_ast_Lip}, \eqref{rho.j.T_lambda} implies
\begin{align} \label{est:nabla.w_lambda}
    \|w_\lambda\|_{W^{1,\infty}} \lesssim 1,
\end{align}
and, due to Lemma \ref{lem.diff.u.W_2.St},
\begin{align}\label{eq:est_dtw}
\begin{aligned}
     \|\partial_t w_\lambda\|_{L^2_{\rho_\lambda}} &\leq \limsup_{h \to 0} \frac{\|w_\lambda(t+h) - w_\lambda(t)\|_{L^2_{\rho_\lambda}}}{h} \\
     &\lesssim  \limsup_{h \to 0} \frac{\W_2(\rho_\lambda(t + h), \rho_\lambda(t))}{h} \\
     &\lesssim \sup_{t < T_\lambda} M_2[f_\lambda(t)],
     \end{aligned}
\end{align}
where the last estimate follows from a computation much as in \eqref{cont.rho}.
Using \eqref{est:nabla.w_lambda}, \eqref{eq:est_dtw}, and Cauchy-Schwarz in $L^2_{f_\lambda(t)}(\R^3)$, we discover 
\begin{align}\label{eq:R}
    R \lesssim M_2[f_\lambda(t)]^{\frac 12}S^{\frac 1 2}\ls S^{\frac 12},
\end{align}
where we used \eqref{rho.j.T_lambda} and $f_\lambda\in \mP(\R^3\times \R^3)$ in the last step. Inserting \eqref{eq:D} and \eqref{eq:R} into \eqref{eq:DR}, we conclude
\begin{align}
    \frac{\dd}{\dd t}S \leq - c \lambda  S +C  S^{\frac 1 2},
\end{align}
where $c,C$ are constants that depend only on $T$, $K$, $M_K[f^0]$, $\|f^0\|_\infty$  and $C_\ast$.
Thus by a comparison principle for the differential inequality for $S{^\frac 12}$
we obtain (with a different constant $C$)
\begin{align} \label{est:S}
    S(t) \leq S(0) e^{-c \lambda t} + \frac {C} {\lambda^2}.
\end{align}

\medskip
We estimate $\W_2(\rho_\lambda,\rho_\ast)$ by
\begin{align} \label{W_2.Z}
    \W_2^2(\rho_\lambda,\rho_\ast) \leq \int |X_\lambda - X_\ast|^2 \dd f^0 \eqqcolon 2 Z
\end{align}
where $X_\ast(t,x)$ are the characteristics associated with $\rho_\ast$, i.e.
\begin{align}
    \dot X_\ast(t,x) = g+ u_\ast(t, X_\ast(t,x)) , \qquad X_\ast(0,x) = x.
\end{align}
Then, we have
\begin{align}
    \frac{\dd}{\dd t} Z &= \int (X_\lambda - X_\ast) \cdot (V_\lambda - g - u_\ast \circ X_\ast) \dd f^0  \\
    &= \int (X_\lambda - X_\ast) \cdot \left((V_\lambda - g - w_\lambda \circ X_\lambda) + (w_\lambda \circ X_\lambda - u_\ast \circ X_\lambda) +  (u_\ast \circ X_\lambda - u_\ast \circ X_\ast)\right) \dd f^0 \\
    & \leq  Z^{\frac 1 2} \left(S^{\frac 1 2} + \|w_\lambda - u_\ast\|_{L^2_{\rho_\lambda}} + \|\nabla u_\ast\|_\infty Z^{\frac 1 2} \right). 
\end{align}
By  \eqref{W_2.Z}, \eqref{est:w_lambda.u_ast}, and \eqref{est:S} we deduce with $c$ as in \eqref{est:S},
\begin{align}
    \frac{\dd}{\dd t} Z \lesssim Z +  Z^{\frac 1 2} \left(S^{\frac 1 2}(0) e^{-\frac 12c \lambda t} + \frac 1 {\lambda} \right).
\end{align}
A comparison principle for the differential inequality for $Z^{1/2}$ and $Z(0) = 0$ yields
\begin{align} \label{W_2.rho_lambda.rho}
    \W_2(\rho_\lambda(t),\rho_\ast(t)) \lesssim Z^{\frac 1 2} (t)\lesssim \frac 1 \lambda (1 + S^{\frac 1 2 }(0)) \lesssim \frac 1 \lambda, 
\end{align}
since $S(0)$ is controlled through $M_K[f^0]$ by using the triangle inequality and \eqref{est:nabla.w_lambda}.
We conclude this step by observing firstly that \eqref{W_2.rho_lambda.rho} is \eqref{W.rho.hydrodynamic}.
Secondly, for \eqref{W.f.hydrodynamic} we consider the transport plan $\gamma_t \in \Gamma(f_\lambda(t), \rho_\ast(t) \otimes \delta_{g + u_\ast(t)})$ defined as $\gamma_t = ((X_\lambda,V_\lambda)(t;0,\cdot), X_\ast(t;0,\cdot), g+ u_\ast(t, X_\ast(t;0,\cdot))_{\#} f^0$ to estimate
\begin{align}
    \W_2^2(f_\lambda, \rho_\ast \otimes \delta_{g + u_\ast}) &\leq Z + \int |V_\lambda - g - u_\ast\circ X_\ast|^2 \dd f^0 \\
    & \lesssim  Z + S + \int |w_\lambda \circ X_\lambda - u_\ast \circ X_\ast|^2 \dd f^0(x, v)\\
    & \lesssim Z + S + Z \|\nabla w_\lambda\|_\infty + \|w_\lambda - u_\ast\|^2_{L^2_{\rho_\ast}} \\
    &\lesssim Z + S,
\end{align}
where we used \eqref{est:nabla.w_lambda} and \eqref{est:w_lambda.u_ast}
in the last estimate. Inserting estimates \eqref{est:S} and \eqref{W_2.rho_lambda.rho} yields the estimate for the Wasserstein distance in \eqref{W.f.hydrodynamic}.
To estimate the velocities in \eqref{W.f.hydrodynamic}, we split $u_\lambda - u_\ast=(u_\lambda-w_\lambda)+(w_\lambda-u_\ast)$. The estimate for $w_\lambda-u_\ast$ is \eqref{est:w_lambda.u_ast}. The estimate for $u_\lambda-w_\lambda$ follows from \eqref{split.u-u_ast} and
\begin{align} \label{u_lambda.w_lmabda.S}
\begin{aligned}
    \|u_\lambda - w_\lambda\|_{L^2(B_1(x))} &=  \|\Br^{-1}_{\rho_\lambda}[\mV_\lambda - g - w_\lambda]\|_{L^2(B_1(x))} \\
    &\lesssim \|\nabla \Br^{-1}_{\rho_\lambda}[\mV_\lambda - g - w_\lambda]\|_{2} \lesssim \|\mV_\lambda - g - w_\lambda\|_{L^2_{\rho_\lambda}} \lesssim S^{\frac 1 2},
\end{aligned}
\end{align}
where we used \eqref{est:A_rho} in the second estimate and \eqref{eq:pythagoras} in the last.

 \medskip

\noindent\textbf{Step 3:} \emph{Proof that $T_\lambda = T$ for $\lambda \gg 1$.}

We show here that after an initial layer
of order $1/\lambda$,  $u_\lambda - u_\ast$ is small which gives control on $\|u_\lambda\|_{W^{1,\infty}}$.
More precisely, we show that there exist $\tilde C \lesssim 1,\tilde c \gtrsim 1$   such that
\begin{align} \label{condition.S_0.t}
    t \in [0,T_\lambda], ~  \frac {\tilde C} \lambda  + \tilde C  S(0) e^{-\tilde c \lambda t} \leq 1  \quad \Longrightarrow  \quad \|u_\lambda(t)\|_{W^{1,\infty}} \leq  \|u_\ast(t)\|_{W^{1,\infty}} + 1  \leq C_\ast.
\end{align}
From this, we conclude $T_\lambda = T$ for $\lambda$ large enough, provided that
\begin{align} \label{T_long}
    T_{\lambda} \geq \frac{\max\{0, \log( \tilde C S(0))\}}{ \tilde c \lambda}.
\end{align}
In case of Condition \ref{cond.well.prepared},   \eqref{T_long} is automatically satisfied since then the right-hand side vanishes. 
In case of conditions  Conditions \ref{cond:L^1L^infty}--\ref{cond:f.small},  Lemma \ref{lem:short.times} implies $T_\lambda \gtrsim 1$ and thus \eqref{T_long} also holds for $\lambda$ large enough.

It remains to prove \eqref{condition.S_0.t}.
For times $t<T_\lambda$ we split again $u_\lambda - u_\ast=(u_\lambda-w_\lambda)+(w_\lambda-u_\ast)$. On the one hand, by interpolation between $\|\nabla^2(w_\lambda - u_\ast)\|_{1+ K/3} \lesssim \|\rho_\ast\|_{1 + K/3} + \|\rho_\lambda\|_{1 + K/3}$ and \eqref{est:w_lambda.u_ast}, and by using \eqref{rho.j.T_lambda}, we have for some universal $\theta > 0$ which we allow to change from line to line in the following,
\begin{align}\label{eq:Lip_w_u}
    \|w_\lambda - u_\ast\|_{W^{1,\infty}} \lesssim \W_2(\rho_\lambda,\rho_\ast)^\theta.
\end{align}
On the other hand, 
using \eqref{eq:Brinkman_est4} and \eqref{rho.j.T_lambda} yields
\begin{align}
    \|\nabla^2(u_\lambda - w_\lambda)\|_{\frac{3+K}4} \lesssim \|j_\lambda\|_{L^{\frac{3+K}4} \cap L^{\frac 65}}\left(1 + \|\rho_\lambda\|_{\frac{3+K}4}\right) \lesssim 1,
\end{align}
where we used that we control $\|j_\lambda\|_{6/5} \lesssim M_{6/5}[f_\lambda] \lesssim (M_K[f_\lambda])^{6/(5K)} \lesssim 1$ by \eqref{moments.T_lambda}.
Hence, interpolation with \eqref{u_lambda.w_lmabda.S} delivers 
\begin{align}\label{eq:Lip_u_w}
    \|u_\lambda - w_\lambda\|_{W^{1,\infty} }\lesssim S^\theta.
\end{align}
Inserting estimates \eqref{est:S} and \eqref{W_2.rho_lambda.rho} into the combination of \eqref{eq:Lip_w_u} and \eqref{eq:Lip_u_w} we obtain
\begin{align}
\|u_\lambda(t) - u_\ast(t)\|_{W^{1,\infty}} \lesssim  S(t)^\theta + \W_2(\rho_\lambda(t),\rho_\ast(t))^\theta 
 \lesssim    \left(\frac 1 \lambda + S(0) e^{-c \lambda t}\right)^\theta.
\end{align}
This yields \eqref{condition.S_0.t} which finishes the proof.
\end{proof}

%% file: 6_Stability.tex
\section{Perturbative derivation of the Vlasov-Stokes equation}
\label{sec:derivation}

This section is devoted to the proof of Theorem \ref{thm:vlasov}. As explained in Subsection \ref{sec:strategy}, the proof is based on an improved stability estimate for the Vlasov-Stokes equation in the regime $\lambda \gg 1$.

\subsection{Improved stability estimate for the Vlasov-Stokes equation for large \texorpdfstring{$\lambda$}{lambda}}
\label{sec:stability.macro}

\begin{prop} \label{pro:stability}
    Let $K>9$, $\lambda > 0$ and  $T> 0$ and $q > 4$. For $i=1,2$, let $f_i^0 \in \mP(\R^3 \times \R^3) \cap L^\infty(\R^3 \times \R^3)$ with $M_K[f_i^0] < \infty$ and let $(f_i,u_i)$ be the unique  solutions to the Vlasov-Stokes equation \eqref{eq:Vlasov.Stokes} on $[0,T]$ with initial data $f_i^0$ provided by Theorem \ref{thm:existence}.
    Let $(X_i,V_i)$ be the associated characterisitics and
    let $\gamma_0 \in \Gamma(f_1^0,f_2^0)$ be a transport plan.
    Set
    \begin{align}
        \m H(t) \coloneqq 
        \frac 1 2 \int |X_1(t;0,x_1,v_1) - X_2(t;0,x_2,v_2)|^2 \dd  \gamma_0 (x_1,v_1,x_2,v_2), 
    \\
        \mathcal E(t) \coloneqq  \frac 1 2 \int  |V_1(t;0,x_1,v_1) - V_2(t;0,x_2,v_2)|^2 \dd  \gamma_0(x_1,v_1,x_2,v_2). 
    \end{align}
    Then, for all $t \leq T$ it holds that
    \begin{align}
        \frac {\dd}{\dd t} \m H(t) &\leq 2 \mathcal E^{\frac 1 2}(t) \m H^{\frac 1 2}(t), \label{H.dot.limit} \\
        \frac {\dd}{\dd t} \mathcal E(t) &\leq \lambda\left(-c \mathcal E(t) +  C \mathcal E^{\frac 1 2}(t) \m H^{\frac 1 2}(t)\right),\label{E.dot.limit}
    \end{align} 
    where $c>0$  and $C<\infty$ depend only on $\|\rho[f_1](t)\|_{q}$, $\|\rho[f_2](t)\|_{q}$, $\|u_2(t)\|_{W^{1,\infty}}$, and $M_9[f_2(t)]$.
\end{prop}
\begin{cor}\label{cor:stability}
    Under the assumptions of Proposition \ref{pro:stability}, for all $t \in [0,T]$,
    \begin{align} \label{stability.cor}
        \W_2(f_1(t),f_2(t)) \leq C \W_2(f_1^0,f_2^0) e^{C t}
    \end{align}
    where the constant $C$ depends only on 
    \begin{align}
        \sup_{t \in [0,T]} \left(\|\rho[f_1](t)\|_{q} +  \|\rho[f_2](t)\|_{q} +  \|u_2(t)\|_{W^{1,\infty}} + M_9[f_2(t)] \right).
    \end{align}
\end{cor}
\begin{proof}
    This is a direct consequence of Proposition \ref{pro:stability},  Lemma \ref{l:ODE} (applied to ${\m H}^{1/2},{\m E}^{1/2}$) and 
    \begin{align}
        \m W_2^2(f_1,f_2) \lesssim \m E + \m H, && \m W_2^2(f_1(0),f_2(0)) = 2 \m E + 2 \m H
    \end{align}
    for an optimal transport plan $\gamma_0$.
\end{proof}
\begin{rem}
    \begin{enumerate}[(i)]
        \item In contrast to the stability estimate in Theorem \ref{thm:existence}, the constant $C$ in \eqref{stability.cor} does not directly depend on $\lambda$. A priori, though, we only control the quantities $\|\rho_1(t)\|_{q}$, $\|\rho_2(t)\|_{q}$,  $\|u_2(t)\|_{W^{1,\infty}}$,  $M_9[f_2(t)]$ with a dependence on $\lambda$ (through \eqref{moments.thm} and Lemma \ref{lem:moments}).
       However, in the setting of Theorem \ref{thm:hydrodynamic.limit}, estimate \eqref{uniform.bounds.thm} does provide a uniform control of these quantities for sufficiently large $\lambda$.
       \item An improvement of \eqref{stability.cor} that is sharp at $t=0$ can be obtained by tracking the terms  appearing in Lemma \ref{l:ODE}.
    \end{enumerate}
\end{rem}

\begin{proof}[Proof of Proposition \ref{pro:stability}]
We will abbreviate $\rho_1=\rho[f_1],\rho_2=\rho[f_2]$ in the following. Estimate \eqref{H.dot.limit} is an immediate consequence of the characteristic equation $\dot X_i = V_i$ and the Cauchy-Schwarz inequality. It remains to prove \eqref{E.dot.limit}.

\medskip

\noindent\textbf{Step 1:} \emph{Reduction to $t=0$.}
We have
    \begin{align}
        \frac {\dd}{\dd t} \m E(t) = \frac{\dd}{\dd s}|_{s=0} \m E(t+s) &= \frac{\dd}{\dd s}|_{s=0} \int |V_1(t+s;0,x_1,v_1)  - V_2(t+s;0,x_2,v_2)| \dd \gamma_0 \\
        &=  \frac{\dd}{\dd s}|_{s=0} \int |V_1(t+s;t,x_1,v_1)  - V_2(t+s;t,x_2,v_2)| \dd \gamma_t,
    \end{align}
    where $\gamma_t$ is the pushforward of $\gamma_0$ under the characteristic flow, i.e 
    \begin{align}
        \gamma_t(x_1,v_1,x_2,v_2)= ((X_1,V_1)(0;t,x_1,v_1), (X_2,V_2)(0;t,x_2,v_2))\# \gamma_0(x_1,v_1,x_2,v_2).
    \end{align}
    Similarly,
    \begin{align}
        \m H(t) = \frac 1 2 \int |x_1 - x_2| \dd \gamma_t, &&   \m E(t) = \frac 1 2 \int |v_1 - v_2| \dd \gamma_t.
    \end{align}
Hence, it suffices to show \eqref{E.dot.limit} at $t=0$.

\medskip

\noindent\textbf{Step 2:} \emph{Decomposition.}
We compute
 \begin{align}
        \frac {\dd}{\dd t} \m E(0) = \lambda \int (u_1(x_1) - u_2(x_2) - (v_1 - v_2))(v_1 - v_2) \dd \gamma 
    \end{align}
and we split as in Step 1 of the proof of Lemma \ref{lem.diff.u.W_2}
    \begin{align}
        u_1 - u_2 = w + (r_1 - r_2) + (r_2 - u_2),
    \end{align}
    with $w,r_1,r_2 \in \dot H^1(\R^3)$ being the solutions to
       \begin{align}
        -\Delta w + \nabla p &= \int (v_1 - v_2 - w) \gamma(\cdot,\dd v_1, \dd x_2, \dd v_2), \quad \dv w = 0, \label{eq:wdef}\\
        -\Delta r_1 +  \nabla p &= \int (v_2 - r_1) \gamma(\cdot,\dd v_1, \dd x_1, \dd v_2), \quad \dv r_1 = 0, \\
    -\Delta r_2 + \nabla p &= \int (v_2 - u_{h_2}) \gamma(\cdot,\dd v_1, \dd x_2, \dd v_2), \quad \dv r_2 = 0.
 \end{align}
Hence, by the Cauchy-Schwarz inequality
 \begin{align} 
        \frac {\dd}{\dd t} \m E|_{t=0} &\leq \lambda \int (w(x_1) - (v_1 - v_2))(v_1 - v_2) \dd \gamma  \\
        &\quad+  \left(\|r_1- r_2\|_{L^2_{\rho_1}} + \|r_2 - u_2\|_{L^2_{\rho_1}} + \sqrt {2 \m H}\|\nabla u_2\|_{\infty}  \right)\sqrt {2\m E}    
        \label{E.splitting}
\end{align}

\noindent\textbf{Step 3:} \emph{Proof that
\begin{align} \label{w.coercive}
    -\int (w(x_1) - (v_1 - v_2))(v_1 - v_2) \dd \gamma
    &\gs \left( 1 + \|\rho_1\|_{3/2}\right)^{-1} \m E.
\end{align}
}
We have
\begin{align}
     \int (w&(x_1) - (v_1 - v_2))(v_1 - v_2)\dd \gamma \\
     &= - \int (w(x_1) - (v_1 - v_2))(w(x_1) - (v_1 - v_2))  \dd \gamma  -  \int ((v_1 - v_2) - w(x_1))w(x_1) \dd \gamma  \\
     &= -\int (w(x_1) - (v_1 - v_2))(w(x_1) - (v_1 - v_2))  \dd \gamma  - \| \nabla w\|_{2}^2,
    \label{identity.w.1}
\end{align}
where in the last equality we used \eqref{eq:wdef} tested with $w$.
We introduce the notation $\mV_1,\tilde {\mV}_2 \in L^2_{\rho_1}(\R^3)$ by
\begin{align}
   \int \mV_1 \psi \dd \rho_1 &\coloneqq \int \psi(x_1) v_1 \dd \gamma \quad \text{for all } \psi \in L^2_{\rho_1}(\R^3),\\
    \int \tilde {\mV}_2  \psi \dd \rho_1 &\coloneqq \int \psi(x_1) v_2  \dd \gamma \quad \text{for all } \psi \in L^2_{\rho_1}(\R^3).
\end{align}
Note that $\mV_1, \tilde {\mV}_2$ are well defined by the Cauchy-Schwarz inequality and by $M_2[f_i] < \infty$ for $i = 1,2$ due to \eqref{moments.thm}. Furthermore $\mV_1=\mV[f_1]$.

By definition, for all $\psi \in L^2_{\rho_1}(\R^3)$, it holds that
\begin{align}
    \int \psi(x_1) (v_1 - v_2 - (\mV_1(x_1) - \tilde{\mV}_2(x_1) )) \dd \gamma = 0.
\end{align}
Applying this to $w$ and $\mV_1 - \tilde \mV_2$, we deduce 
\begin{align}\label{identity.w.2}
\begin{aligned}
    \int (w(x_1) - (v_1 - &v_2))(w(x_1) - (v_1 - v_2))  \dd \gamma \\
    &= \int (w(x_1) - (\mV_1 - \tilde{\mV}_2 ))(w(x_1) - (v_1 - v_2)) \dd \gamma \\
    &\quad + \int(v_1 - v_2 - (\mV_1 - \tilde{\mV}_2 ))(w(x_1) - (v_1 - v_2)) \dd \gamma  \\
    &= \int |w(x_1) - (\mV_1 - \tilde{\mV}_2 )|^2 \dd \rho_1 + \int|v_1 - v_2 - (\mV_1 - \tilde{\mV}_2 )|^2\dd \gamma.
    \end{aligned}
\end{align}
Combining identities \eqref{identity.w.1} and \eqref{identity.w.2} and using  Lemma \ref{lem:frictionCoercive}, we deduce
\begin{align}
    \int (w(x_1) - (v_1 - v_2))&(v_1 - v_2) \dd \gamma \\
    &\ls  -\left( 1 + \|\rho_1\|_{3/2}\right)^{-1}  \|\mV_1 - \tilde{\mV}_2\|_{L^2_{\rho_1}}^2 - \int(v_1 - v_2 - (\mV_1 - \tilde{\mV}_2 ))^2\dd \gamma \\
    &\ls -\left( 1 + \|\rho_1\|_{3/2}\right)^{-1}\m E
\end{align}
as claimed, where we used in the last estimate as before that 
\begin{align}
    \int (v_1-v_2-(\mV_1-\tilde{\mV}_2))(\mV_1-\tilde{\mV}_2)\dd \gamma=0,
\end{align}
to add the squares.
\medskip

\noindent\textbf{Step 4:} \emph{Estimate of the remainders and conclusion.}
By the estimates \eqref{est:Step4} and \eqref{est:Step3} of the proof of Lemma \ref{lem.diff.u.W_2}, we have
\begin{align} \label{r.phi.H}
    \| r_1 - r_2\|_{L^2_{\rho_1}}^2 + \|r_2 - u_2\|_{L^2_{\rho_1}}^2 \lesssim \|r_2 - u_2\|_{L^2_{\rho_1}}^2 \lesssim  \mH.
\end{align}
Inserting estimates \eqref{w.coercive},  and \eqref{r.phi.H} into \eqref{E.splitting} yields \eqref{E.dot.limit}.
\end{proof}

\subsection{Adaptation to the microscopic system -- proof of Theorem \ref{thm:vlasov}}
\label{sec:Stability.micro}

In this subsection we will always work under the assumptions of Theorem \ref{thm:vlasov}. In particular, all the statements are implicitly understood to hold for sufficiently large $N$ and the implicit constant in the notation $A \lesssim B$ is allowed to depend on the same quantities as the constant $C$ in Theorem \ref{thm:vlasov}. 

Throughout this subsection we will also use information on the evolution coming from \cite{HoferSchubert23}. In particular we have the following estimates for $t\le T$ and $N$ large enough.
\begin{align}
    \dmin(t) &\gtrsim \dmin(0)e^{-Ct} \gg N^{-\frac{59}{153}}, \label{dmin.unif} \\
    \frac 1 N (|V|_9^9 + |N F|_9^9) &\lesssim  1, \label{moments.V.F}\\
        \frac 1N(S_1 + S_{9/4})&\lesssim 1\label{eq:S_est},
\end{align}
where $S_k$ and $F_i$ are defined in \eqref{def.S} and \eqref{def:F_i}, respectively, and where we recall the notation \eqref{euclidean}.

Indeed, \eqref{dmin.unif} follows from \eqref{dmin.thm}. Moreover, \cite[Remark 1.2(viii)]{HoferSchubert23} states that $  |N F_i(t)| + \abs{V_i(t)} \ls 1 + |V_i^0|$ for all $1 \leq i \leq N$. Combining this with \eqref{ass:V.norms} implies \eqref{moments.V.F}.
Finally, by \cite[Lemma 2.3]{HoferSchubert23}, we have
\begin{align} \label{eq:sums.Wasserstein.p}
	\frac 1NS_\beta\lesssim \frac{1}{N\dmin^\beta}+\|\rho_\ast\|_{q}^{\frac{\beta q'}3}  &+ 
 \frac{\|\rho_\ast\|_{q}^{\frac{3-\beta}3 \frac {2 q'} {3+2 q'}}}{N^{\beta/3} \dmin^{\beta}} (\W_2(\rho_N,\rho_\ast))^{\frac{2(3-\beta)}{3+2q'}} \\
	&+ \left(\frac{\|\rho_\ast\|_{q}^{\frac{3-\beta}3 }}{N^{\beta/3} \dmin^{\beta}}\right)^{\frac{(\beta+2)q'}{3 - \beta + \beta q'+2q'}} (\W_2(\rho_N,\rho_\ast))^{\frac{2(3-\beta)}{3 - \beta + \beta q'+2q'}}.
\end{align}
We apply this with $\beta \leq 9/4$, and $\rho_\ast$ the solution to \eqref{eq:transport-Stokes} from Theorem \ref{th:diagonal} which satisfies $\|\rho_\ast(t)\|_q = \|\rho_0\|_q$. Combining moreover with  \eqref{ass:W.lambda}, \eqref{W_2.simplified}, \eqref{dmin.threshold} and \eqref{dmin.unif} yields \eqref{eq:S_est}.  To see that the last term in the right-hand side of \eqref{eq:sums.Wasserstein.p} vanishes as $N\to \infty$, it is enough to raise it to the power $\frac{3-\beta+\beta q'+2q'}{3+2q'}$, and to see that the power $\frac{(\beta+2)q'}{3+2q'}$ of $N^{-\beta/3}\dmin^{-\beta}$ is smaller than $1$ for $\beta \leq 9/4$ due to the condition $q>4$.

Note that, by Jensen's inequality, \eqref{eq:S_est} implies analogous bounds for $S_k$, $1 \leq k \leq 9/4$, and similarly \eqref{moments.V.F} implies bounds for the analogous quantities with power $<9$. For the statements in this section we will simplify all estimates based on \eqref{eq:S_est} and \eqref{moments.V.F}. We will, however only simplify in the last step of the proofs so that the full dependencies can be extracted easily from the proofs.

\medskip

For $4 R< d \leq \dmin(t)/6$, we denote
\begin{align}
    B^d_i(t) &\coloneqq B_d(X_i(t)), \label{B_i^d} \\
    f_N^d(t,\cdot) &\coloneqq \frac 1 N \sum_i \delta_{B_i^d(t)} \otimes \delta_{V_i(t)},\label{rho_N^d} \\
    \rho_N^d(t,\cdot) &\coloneqq \rho[f_N^d(t,\cdot)] =  \frac 1 N \sum_i\delta_{B_i^d(t)}.
\end{align}

We introduce
$w_N^d$, the solution to
\begin{align}\label{eq:w_N}
    -\Delta w_N^d(t,\cdot) + \nabla p = \sum_i F_i(t) \delta_{B_i^d(t)}, \quad \dv w_N^d = 0,
\end{align}
and will typically suppress the time-dependence in the following.

We show first that $w_N^d$ is a good approximation for $u_N$ and moreover a well-behaved function.
To this end, an important ingredient from \cite{HoferSchubert23} is the characterization of the forces $F_i$: From Lemma \ref{le:force.representation} it follows that 
\begin{align} \label{force.repr}
    F_i = 6 \pi R\left(V_i-\fint_{A_i^d} \omega_i u_N \dd x\right),
\end{align}
where $ A_i^d \coloneqq B_{d}(X_i) \setminus B_{d/2}(X_i)$ and $\omega_i(x) \coloneqq \omega(x - X_i)$ with $\omega$ as in Lemma \ref{le:force.representation}.

\begin{lem} \label{lem:F.w_N}
    Let $w_N^d$ be as in \eqref{eq:w_N}. Then, 
    \begin{align} 
            \|w_N^d \|_{W^{1,\infty}(\R^3)} &\lesssim\left(1 + \frac 1 {N^{8/9} d^{2}} \right), \label{w_N.Lipschitz} \\
            \label{eq:u.w}
      \sup_i  \left|\fint_{A_i^d} \omega_i (u_N - w_N^d) \dd x\right|+\left|\fint_{A_i^d} \omega_i u_N \dd x - w_N^d(X_i)\right| &\lesssim d\left(1 + \frac 1 {N^{8/9} d^{2}} \right), \\
        \label{eq:u.w.L^2_loc}  \sup_{x \in \R^3} \|u_N - w_N^d \|_{L^2(B_1(x))} &\lesssim N^{-1/2} d^{1/2} + N^{-5/3}.
    \end{align}
\end{lem}
\begin{proof} \textbf{Step 1:} \emph{Proof of \eqref{w_N.Lipschitz}.}
We split $w_N^d = \sum_i w_{N,i}$, where $w_{N,i} \in \dot H^1(\R^3)$ is the solution to
\begin{align}
         -\Delta w^d_{N,i} + \nabla p = F_i \delta_{B_i^d}, \quad \dv w^d_{N,i} = 0.
\end{align}
Then for $k = 0,1$ and all $x \in \R^3$ (e.g. by a scaling argument or explicit convolution with the fundamental solution)
\begin{align}
    |\nabla^k w^d_{N,i}(x)| \lesssim \frac{\abs{F_i}}{|x - X_i|^{1+k} + d^{1+k}}.
\end{align}

Using \eqref{eq:S_est} and \eqref{moments.V.F} implies         \begin{align}
        | w_N^d(x)| & \lesssim \sum_i \frac{\abs{F_i}}{|x - X_i| + d} \le \left( \frac 1N S_{9/8} +\frac{1}{Nd^{9/8}}\right)^{8/9 }  N^{-1/9} |N F|_9 \ls 1+\frac{1}{N^{8/9}d},\\
        |\nabla w_N^d(x)| & \ls \sum_i \frac{| F_i|}{|x - X_i|^2 + d^2} \le  \left( \frac 1 N S_{9/4}+\frac{1}{Nd^{9/4}}\right)^{8/9}N^{-1/9} |N F|_9 \ls 1+\frac{1}{N^{8/9}d^2}. \quad \label{eq:nabla_w}
    \end{align}
Since $d < \dmin \to 0$ this finishes the proof of \eqref{w_N.Lipschitz}. \\

\noindent \textbf{Step 2:} \emph{Proof of \eqref{eq:u.w}.}
Using \eqref{omega}, we have
    \begin{align}\label{eq:force_rep_triangle}
        \left|\fint_{A_i^d} \omega_i u_N \dd x - w_N^d(X_i)\right| \ls  \left|\fint_{A_i^d} \omega_i(x) (w_N^d(X_i) - w_N^d(x))  \dd x\right| +  \left| \fint_{A_i^d} \omega_i (w_N^d - u_N) \dd x \right|,
    \end{align}
    and $\|\omega_i\|_\infty \lesssim 1$.
    In view of \eqref{w_N.Lipschitz}, it thus remains to bound the first term on the left-hand side of \eqref{eq:u.w}.
    
    To this end,  we follow the proof of item (i) in \cite[Lemma 5.3]{HoferSchubert23}.
For given $W\in \R^3$ consider the solution $\varphi$ to
\begin{align}
    -\Delta \varphi+\nabla p= \omega_i^T W \frac{1}{\abs{A_i^d}}\1_{A_i^d}, ~ \dv \varphi=0 \quad \text{in} ~\R^3.
\end{align}
We compute
\begin{align}
\begin{aligned}
    W\cdot\fint_{A_i^d}\omega_i (w_N^d-u_N)\dd x&=\fint_{A_i^d}(w_N^d-u_N)\cdot \omega_i^T W \dd x
    =\int (w_N^d-u_N)\cdot (-\Delta \varphi+\nabla p)\dd x\\
    &= 2 \int \nabla (w_N^d-u_N)\cdot e\varphi\dd x
    = 2 \int e(w_N^d-u_N)\cdot\nabla\varphi\dd x\\
    &=\sum_j \fint_{B^d_j} \varphi\dd x \cdot F_j+\sum_j \int_{\partial B_j} \varphi\cdot (\sigma[u_N]n)\dd \mathcal{H}^2\\
    &=\sum_j \int_{\partial B_j} (\varphi-\varphi_j)\cdot (\sigma[u_N]n)\dd \mathcal{H}^2-\sum_j (\varphi_j-\varphi_j^d)\cdot F_j ,
    \end{aligned}
    \label{eq:comp_test1}
\end{align}
where $\varphi_j\coloneqq\fint_{\partial B_j} \varphi\dd \mathcal{H}^2$ and $\varphi_j^d\coloneqq\fint_{B^d_j} \varphi\dd x$.
Using that $\varphi= \Phi\ast \omega_i^T W \frac{1}{\abs{A_i^d}}\1_{A_i^d}$, the fact that $\dist(A_i^d,A_j^d)\gs d_{ij}$ for $i\neq j$, the second property from \eqref{omega}, and the decay of $\nabla \Phi$, we have for all $j\neq i$ that $|\varphi_j^d - \varphi_j| \lesssim \abs{W}d d_{ij}^{-2}$. Moreover $\|\varphi\|_\infty\ls |W|d^{-1}$ implies the bound $|\varphi_i-\varphi_i^d|\ls |W|dd^{-2}$.
Thus, using again \eqref{eq:S_est} and \eqref{moments.V.F},
     \begin{align}\label{eq:phi_diff}
     \begin{aligned}
       \sum_j |\varphi_j^d - \varphi_j||F_j| &\lesssim d\abs{W}  \left(\frac 1 N S_{9/4}+\frac{1}{Nd^{9/4}}\right)^{8/9}N^{-1/9} |N F|_9\\
       &\ls d\bra{1+\frac{1}{N^{8/9}d^2}}|W|.
       \end{aligned}
    \end{align}
    In order to estimate the first term on the right-hand side of \eqref{eq:comp_test1}, we consider the solution  $w_0$ to 
    \begin{align}\label{eq:tilde_u}
        -\Delta w_0 + \nabla p = \sum_i \delta_{\partial B_i} F_i, ~ \dv w_0 = 0 \quad \text{in} ~ \R^3.
    \end{align}
    Here, $\delta_{\partial B_i}=\frac{1}{|\partial B_i|}\mathcal{H}^2_{|\partial B_i}$ is the normalized Hausdorff measure. Then, by \cite[Lemma 5.4]{HoferSchubert23}, 
    \begin{align} \label{est.u-utilde}
       \|w_0 - u_N\|^2_{\dot H^1(\R^3)} \leq  \| \nabla w_0\|^2_{L^2(\cup_i B_i)} \lesssim R^3  |F|_2^2 S_2^2.
    \end{align}
 Moreover, using the weak formulation for $w_0$ defined as in \eqref{eq:tilde_u}, we have
\begin{align}
    \int_{\R^3} \nabla \psi \cdot e   w_0 \dd x = 0.
\end{align}
By a classical extension result (see e.g. {\cite[Lemma 3.5]{Hofer18MeanField}}), there exists a divergence free function $\psi \in \dot H^1(\R^3)$ such that for all $1 \leq j \leq N$, $\psi = \varphi - \varphi_j$ in $B_j$ and
\begin{align}
		\|\nabla \psi\|^2_{2} \lesssim  \|\nabla \varphi \|^2_{L^2(\cup_i B_i)} \lesssim |W|^2 R^3\left(S_4 + \frac 1 {d^4} \right),
\end{align}
where we used that $\nabla \varphi$ decays like $ \nabla \Phi$. Hence,
\begin{align} \label{eq:double_diff1}
\begin{aligned}
    \left|\sum_j \int_{\partial B_j} (\varphi-\varphi_j)\cdot (\sigma[u_N]n)\dd \mathcal{H}^2\right|  & = \left|2 \int_{\R^3} \nabla \psi \cdot e  (u_N - w_0) \dd x\right|  \ls \norm{\nabla(u-w_0)}_2\norm{\nabla \psi}_2 \\
    & \lesssim R^{3}  |F|_2  S_2 \left(S_4 + \frac 1 {d^4} \right)^{\frac 1 2}   |W|\\
     & \ls  d R^2 \frac{1}{N^{1/2}} N
     \frac{1}{d^2}\abs{W}\le d\bra{1+\frac{1}{N^{8/9}d^2}}|W| ,
     \end{aligned}
\end{align}
where we used $R<d<\dmin$, \eqref{eq:S_est}, $S_4 \lesssim \dmin^{-4}$ (see e.g.  \cite[Lemma 4.8]{NiethammerSchubert19}), and \eqref{moments.V.F} in the second-to-last estimate, and \eqref{ass:gamma} in the last estimate. Combining \eqref{eq:phi_diff} and \eqref{eq:double_diff1} in \eqref{eq:comp_test1} and taking the sup over all $|W|=1$ finishes the proof of \eqref{eq:u.w}.
\\

\noindent \textbf{Step 3:} \emph{Proof of \eqref{eq:u.w.L^2_loc}.}
The proof of \eqref{eq:u.w.L^2_loc} is similar: Fix $x_0 \in \R^3$ and let $h \in L^2(B_1(x_0))$ with $\|h\|_2 \leq 1$. After extending $h$ by zero to a function on $\R^3$, consider the solution $\varphi \in \dot H^1$ to
\begin{align}
    -\Delta \varphi+\nabla p= h, ~ \dv \varphi=0 \quad \text{in} ~\R^3.
\end{align}
By regularity we have
\begin{align}\label{eq:nabla2phi}
    \|\nabla^2 \varphi\|_{2} \lesssim \|h\|_2\le 1.
\end{align}
Analogously to \eqref{eq:comp_test1}, we obtain
\begin{align}
    \int_{B_1(x_0)} h \cdot (w_N^d-u_N)\dd x 
    =\sum_j \int_{\partial B_j} (\varphi-\varphi_j)\cdot (\sigma[u_N]n)\dd \mathcal{H}^2-\sum_j (\varphi_j-\varphi_j^d)\cdot F_j.
    \label{eq:comp_test2}
\end{align}
Denoting $(\nabla \varphi)_j^d$ the corresponding average of the gradient, we observe that by symmetry $\int_{\partial B_j} (\nabla \varphi)_j^d (x - X_j) \dd \mathcal H^2(x) =0$.
We resort to the Poincar\'e-Sobolev inequality
\begin{align}
    \|\psi\|_{L^\infty (B_j^d)}  \lesssim d^{1/2} \|\nabla^2  \psi\|_{L^2(B_j^d)} \qquad \text{for all } \psi \in W^{2,2}(B_j^d) \text{ with } \int_{B_j^d} \psi = 0, ~ \int_{B_j^d} \nabla \psi = 0.
\end{align}
Hence, 
\begin{align}
    |\varphi_j-\varphi_j^d| = \left|\fint_{\partial B_j} \varphi - \varphi_j^d  - (\nabla \varphi)_j^d (x - X_j) \dd \mathcal H^2(x) \right|  \lesssim d^{1/2} \|\nabla^2  \varphi\|_{L^2(B_j^d)},
\end{align}
and using \eqref{eq:nabla2phi}, we estimate 
     \begin{align} \label{phi_diff.2}
       \sum_j |\varphi_j^d - \varphi_j||F_j| \lesssim N^{-1/2} d^{1/2} \|\nabla^2  \varphi\|_{2}  N^{-1/2} |N F|_2 \ls N^{-1/2} d^{1/2}.
    \end{align}
On the other hand, using Soboloev-embedding and again \eqref{eq:nabla2phi} yields
\begin{align}
    \|\nabla \varphi\|_{L^2(\cup_j B_j)} \lesssim \|\nabla \varphi\|_{6} (N R^3)^{1/3} \lesssim N^{1/3} R.
\end{align}
Hence, adapting the argument for \eqref{eq:double_diff1} yields
\begin{align} \label{eq:double_diff2}
\begin{aligned}
     \left|\sum_j \int_{\partial B_j} (\varphi-\varphi_j)\cdot (\sigma[u]n)\dd \mathcal{H}^2 \right| \ls N^{1/3} R^{5/2} S_2 \abs{F}_2 
     \ls  N^{-5/3}.
     \end{aligned}
\end{align}
Inserting  \eqref{phi_diff.2} and \eqref{eq:double_diff2}  into \eqref{eq:comp_test2} yields \eqref{eq:u.w.L^2_loc}.
\end{proof}

Next, we will compare $w_N^d$ to yet another velocity field $\tilde w_N^d$, which solves the Brinkman equation for the smeared out density $f_N^d$. More precisely, let $\tilde w_N^d \in \dot H^1(\R^3)$ be the solution to
\begin{align} \label{eq:w_N^d}
    -\Delta \tilde w_N^d + \nabla p = \int (v - \tilde w_N^d)  f_N^d(\cdot,\dd v), \qquad \dv \tilde w_N^d = 0. 
\end{align}

\begin{lem} \label{lem:A_rho[V]-w_N}
    It holds that
    \begin{align}
        \|\tilde w_N^d - w_N^d\|_{L^2_{\rho^d_N}} + \|\nabla (\tilde w_N^d  - w_N^d)\|_{2} \ls d\bra{1+\frac{1}{N^{8/9}d^2}}.
    \end{align}
\end{lem}
\begin{proof}
    We introduce the short notations
    \begin{align}
        w_{\diff}\coloneqq w_N^d - \tilde w_N^d,\quad
    F_{i,\diff}\coloneqq F_i - 6 \pi R\left(V_i-\fint_{A_i^d} \omega_i w_N^d \dd y\right).
\end{align}
Using \eqref{eq:w_N^d}, \eqref{rho_N^d}, \eqref{eq:w_N} and \eqref{ass:gamma}, we deduce
\begin{align}
    -\Delta w_{\diff} + \nabla p = \frac 1 N \sum_i \left( N F_{i,\diff}- \fint_{A_i^d} \omega_i w_N^d \dd y   + \tilde w_N^d   \right) \delta_{B^d_i} , \qquad     \div w_{\diff} = 0.
\end{align}
We rewrite 
\begin{align}
    -\Delta w_{\diff} + \rho_N^d w_{\diff} + \nabla p =  \frac 1 N \sum_i \left(  NF_{i,\diff} + w_N^d-\fint_{A_i^d} \omega_i w_N^d  \dd y \right) \delta_{B^d_i}.
\end{align}
Testing the equation with $w_{\diff}$ yields
\begin{align} \label{w_diff.energy}
\begin{aligned}
    \|w_{\diff}\|_{L^2_{\rho_N^d}}^2+\|\nabla w_{\diff} \|_{2}^2  =  \frac 1 N \sum_i \fint_{B_i^d}  \left(  N F_{i,\diff} +w_N^d- \fint_{A_i^d} \omega_i w_N^d \dd y \right)  w_{\diff}(x) \dd x \\
    \lesssim \|w_{\diff}\|_{L^2_{\rho_N^d}} \left(\frac 1 N \sum_i  \bra{ N^2|F_{i,\diff}|^2 + \fint_{B_i^d} \left|w_N^d(x)-\fint_{A_i^d} \omega_i w_N^d \dd y \right|^2  \dd x }\right)^{\frac 1 2}.
    \end{aligned}
\end{align}
Note first that in view of \eqref{force.repr} and Lemma~\ref{lem:F.w_N}, 
\begin{align} \label{F.diff}
   \sup_i \abs{F_{i,\diff}}\ls Rd\bra{1+\frac{1}{N^{8/9}d^2}}.
\end{align}
Secondly, using \eqref{omega} and \eqref{w_N.Lipschitz}, we estimate
\begin{align}\label{eq:meanintdiff}
 \frac 1 N  \sum_i \fint_{B_i^d} \left|w^d_N(x)-\fint_{A_i^d} \omega_i w^d_N \dd y\right|^2 \dd x \lesssim d^2\|\nabla w_N^d\|^2_{\infty} \ls d^2\bra{1+\frac{1}{N^{8/9}d^2}}^2.
\end{align}
Inserting \eqref{F.diff} and \eqref{eq:meanintdiff} into \eqref{w_diff.energy} and using \eqref{ass:gamma} concludes the proof.
\end{proof}

\begin{proof}[Proof of Theorem \ref{thm:vlasov}]
     We divide the proof in seven steps. Steps 1--6 are devoted to an adaptation of the proof of Proposition \ref{pro:stability} which then yields the estimate for $\W_2^2(f_N,f_{\ast,\lambda_N})$ in \eqref{eq:est_main}. In Step 7 we  show the estimate on the fluid velocities in \eqref{eq:est_main}. We will in the following drop the subscripts $\ast,\lambda_N$ and just write $f,u$ instead of $f_{\ast,\lambda_N}, u_{\ast,\lambda_N}$.

     \medskip

\noindent\textbf{Step 1:}  \emph{Modulated energy estimates.} 
Let  $\gamma_0 \in \Gamma(f^0,f_N^0)$ be an optimal transport plan for the $\W_2$ distance, i.e.
\begin{align}
    \W_2^2(f^0,f_N^0) = \int |x - x_N|^2 + |v - v_N|^2 \dd \gamma_0(x,v,x_N,v_N).
\end{align}
Then, denoting by $(X,V)$ and $(X^N,V^N)$ the characteristics with respect to the solution $f$ and $f_N$, respectively, and by $\gamma_t=(X,V,X^N,V^N)\#\gamma_0$ the push-forward transport plan, we define
   \begin{align}
        H(t)&\coloneqq 
        \frac 1 2 \int  |X(t;0,x,v) - X^N(t;0,x_N,v_N)|^2 \dd  \gamma_0 (x,v,x_N,v_N)\\
        &=\frac 12\int |x-x_N|^2\dd \gamma_t(x,v,x_N,v_N),
    \\
        E(t)&\coloneqq  \frac 1 2 \int  |V(t;0,x,v) - V_N(t;0,x_N,v_N)|^2 \dd  \gamma_0(x,v,x_N,v_N)\\
        &=\frac 12\int |v-v_N|^2\dd \gamma_t(x,v,x_N,v_N).
    \end{align}
     We claim that for all $t \leq T$ we have
    \begin{align}
        \frac {\dd}{\dd t}  H(t) &\leq 2  E^{\frac 1 2}(t)  H^{\frac 1 2}(t), \label{H.dot} \\
        \frac {\dd}{\dd t} E(t) &\leq \lambda\left(-c  E(t) +  C E^{\frac 1 2}(t)  \left(H^{\frac 1 2}(t) + N^{-4/9}\right)\right),\label{E.dot}
    \end{align} 
    where $c$  and $C$ depend only on the constants from \eqref{ass:gamma}--\eqref{ass:V.norms} and on $\sup_{t \in [0,T]}\|\rho(t)\|_{q}$.

   Combining \eqref{H.dot}--\eqref{E.dot} with Lemma \ref{l:ODE} implies
   \begin{align}\label{eq:EH}
      W_2^2(f_N(t),f(t)) \lesssim E(t) + H(t) \lesssim (E(0) + H(0) + N^{-8/9}) e^{C t} \lesssim  (W_2^2(f_N^0,f^0) + N^{-8/9}) e^{C t} \qquad
   \end{align}
   with $C$ as in the statement. 
Since $\|\rho^0\|_4 \lesssim 1$ (by the dependencies on the implicit constant in the statement),  estimate \cite[(1.16)]{HoferSchubert23} implies that
    \begin{align} \label{lower.discretization.error}
         N^{-8/9} \lesssim W_2^2(\rho^0_N,\rho^0) \lesssim W_2^2(f_N^0,f^0).
    \end{align}
    Combining \eqref{eq:EH} and \eqref{lower.discretization.error} yields the estimate on   $W_2(f_N(t),f(t))$ in \eqref{eq:est_main}.\\

     As in Step 1 of the proof of Proposition \ref{pro:stability}, \eqref{H.dot} is immediate from the  Cauchy-Schwarz inequality, and the characteristic equations $\dot X = V$ and $\dot X^N = V^N$. To finish the proof of \eqref{eq:est_main} it remains to prove \eqref{E.dot}. This is achieved in Step 2--5.

     \medskip

\noindent\textbf{Step 2:} \emph{Proof that
   \begin{align} \label{E.dot.0}
       \frac{\dd}{\dd t}  E(t) \leq  \lambda_N (-c E(t) + \sqrt{2 E} \left(\int \left|r_1(x) -v^N  -  \mf F (x^N)\right|^2 \dd \gamma_t(x,v,x^N,v^N) \right)^{\frac 1 2},
   \end{align}
   with $c$ depending on the quantities stated in the theorem, where
   \begin{align} \label{r_1}
           -\Delta r_1 +  \nabla p = \int (v^N - r_1) \gamma(\cdot,\dd v, \dd x^N, \dd v^N), \quad \dv r_1 = 0, 
   \end{align}
   and where $ \mf F \colon \R^3 \to \R^3$ is any function that satisfies $ \mf F(X_i) = -NF_i$.
   }
   It suffices to prove this claim for $t=0$, since 
   \begin{align}
       E(t+s) &=  \int |V_1(t+s;0,x,v)  - V_2(t+s;0,x,v)| \dd \gamma_0\\
        &=   \int |V_1(t+s;t,x,v)  - V_2(t+s;t,x,v)| \dd \gamma_t,
   \end{align}
   and therefore the proof for $t > 0$ is analogous.

   We use \eqref{eq:acceleration}, \eqref{def:F_i}, and \eqref{ass:gamma}, as well as \eqref{def:char} to compute
   \begin{align} 
     \frac{\dd}{\dd t}  E(0) &= \lambda_N \int\left(v  - v^N\right) \cdot \left( u(x)   - v  - \mf F (x^N)  \right) \dd  \gamma_0(v,x,v^N,x^N) = \lambda(D + I),
\end{align}
where 
\begin{align}
    D & = \int (v - v^N) \cdot \left(u(x) - r_1(x)  -  (v - v^N)\right) \dd  \gamma_0(v,x,v^N,x^N) , \\
    I &= \int  (v - v^N)   \cdot \left(r_1(x) -v^N  -  \mf F (x^N) \right)  \dd  \gamma_0(v,x,v^N,x^N).
\end{align}
The estimate  $D \leq  -c E(t)$ is analogous as in the proof of Proposition \ref{pro:stability} (consider $w=u-r_1$), and accordingly $c$ depends on $\|\rho\|_{3/2}$. 
Applying the Cauchy-Schwarz inequality to $I$ yields the claim.

In the following we split $I=I_1+I_2$ and estimate both parts individually.

\medskip

\noindent\textbf{Step 3:} \emph{Proof that for $d = N^{-4/9}$
\begin{align} 
     I_1 &\coloneqq \int \left|w_N^d(x^N) -v^N  -  \mf F (x^N)\right|^2 \dd \gamma_t(x,v,x^N,v^N)
     =\frac 1N\sum_i |w_N^d(X_i)-V_i-\mf F (X_i)|^2 \lesssim N^{-8/9}.\label{I_1}
\end{align}}
Note that choosing $d = N^{-4/9}$ the condition $d < \dmin/6$ is satisfied for $N \gg 1$ since
$N^{-4/9} \ll \dmin$ due to \eqref{dmin.unif}. 

By \eqref{force.repr} and \eqref{ass:gamma}, we have 
\begin{align}
    V_i + \mf F(X_i) = V_i -  N F_i = \fint_{A_i^d} \omega_i u_N \dd x,
\end{align}
and thus, estimate \eqref{eq:u.w} implies \eqref{I_1}.

\medskip

\noindent\textbf{Step 4:} \emph{Proof that  for $d = N^{-4/9}$
\begin{align}  \label{I_2}
    I_2 \coloneqq \int \left|r_1(x) - w_N^d(x^N) \right|^2 \dd \gamma_t(x,v,x^N,v^N) \lesssim  H+N^{-8/9}    
    + \|r_2 - \tilde w_N^d\|_{L^2_{\varrho}}^2,
\end{align}
where $r_2$ is defined in \eqref{r_2} below.
} Let $\tilde w_N^d$ be as in \eqref{eq:w_N^d}.
Then, we further split
\begin{align}
    I_2 &= \int \left|r_1(x) - w_N^d(x^N) \right|^2 \dd \gamma_t(x,v,x^N,v^N)\\
     &=\int \left|r_1(x) - \tilde w_N^d(x) +\tilde w_N^d(x) -w_N^d(x) + w_N^d(x) - w_N^d(x^N) \right|^2 \dd \gamma_t(x,v,x^N,v^N)  \\
     & \lesssim I_2^1 + I_2^2 + I_2^3,
\end{align}
where
\begin{align}
     I_2^1 &\coloneqq\left\|r_1- \tilde w_N^d\right\|_{L^2_{\varrho}}^2 + \sup_{x\in \R^3}\left\|r_1- \tilde w_N^d\right\|_{L^2(B_1(x))}^2, \label{def.I_2^1} \\
     I_2^2&\coloneqq   \left\|\tilde w_N^d - w_N^d\right\|_{L^2_{\varrho}}^2 + \sup_{x\in \R^3}\left\|\tilde w_N^d - w_N^d\right\|_{L^2(B_1(x))}^2  ,\label{def.I_2^2} \\
   I_2^3&\coloneqq \int \left| w_N^d(x) - w_N^d(x^N) \right|^2 \dd \gamma_t(x,v, x^N,v^N).
\end{align}
We remark that the second terms on the right-hand side in both \eqref{def.I_2^1} and \eqref{def.I_2^2} are not needed at this point but are included for future reference only (for the proof of the velocity estimate in Step 7).
By Lemma \ref{lem:A_rho[V]-w_N} and a Sobolev inequality, we have
\begin{align} \label{I_2^2}
    I_2^2 \lesssim N^{-8/9}.
\end{align}
Moreover, by  \eqref{w_N.Lipschitz},
\begin{align} \label{I_2^3}
    I_2^3  \lesssim H.
\end{align}
In order to estimate $I_2^1$, we introduce $r_2 \in \dot H^1(\R^3)$ as the solution to 
\begin{align} \label{r_2}
  -\Delta  r_2 + \nabla p = \int (v^d - \tilde w_N^d) \dd \gamma^d_t(\cdot,v,x^d,v^d), \qquad \dv r_2 = 0,
\end{align}
where $\gamma^d_t \in  \Gamma(f,f^d_N)$ is constructed as follows:
We consider the
transport map  $T^d$ from $\rho_d$ to $\rho_N$ defined through $T_d(y) \coloneqq X_i$ in $B_i^d$ and the induced  plan 
$\gamma_t^{N,d} \coloneqq (T_d,\Id,\Id,\Id)\# f_N^d \in \Gamma(f_N,f_N^d)$.
Then, by the Gluing Lemma (cf. \cite[Lemma 5.5]{Santambrogio15}), there exists a measure $\sigma \in \mathcal P((\R^3)^6)$ such that $(\pi_{1,2,3,4})\# \sigma = \gamma_t$ and $(\pi_{3,4,5,6})\# \sigma = \gamma_{N,d}$, where  $\pi$ denotes the projections on the indicated coordinates. We then denote $\gamma_t^d \coloneqq(\pi_{1,2,5,6})\# \sigma \in \Gamma(f,f_N^d)$.
Since $\gamma_t^{N,d}$ is supported on $\{(x^N,v^N,x^d,v^d) : v^N = v^d\}$, we have
\begin{align}
    \int |v^d - v^N|^2 \dd \sigma(x,v,x^N,v^N,x^d,v^d) = \int |v^d - v^N|^2 \dd \gamma_t^{N,d}(x^N,v^N,x^d,v^d) = 0,
\end{align}
and thus $\sigma$ is supported on $\{(x,v,x^N,v^N,x^d,v^d) : v^N = v^d\}$.
In particular, it holds for all $\psi \in L^2_\rho(\R^3)$ that
\begin{align}
    \int \psi(x) v^N \dd \gamma_t(x,v,x^N,v^N) &= \int \psi(x) v^N \dd \sigma(x,v,x^N,v^N,x^d,v^d) \\
    &= \int \psi(x) v^d \dd \sigma(x,v,x^N,v^N,x^d,v^d) = \int \psi(x) v^d \dd \gamma^d_t(x,v,x^N,v^N).
\end{align}
Therefore, by the defining equations of $r_1$ and $r_2$, \eqref{r_1} and \eqref{r_2},
\begin{align} \label{r_1-r_2}
    -\Delta (r_1 - r_2) + \rho(r_1 - r_2) =  \int_{(\R^3)^4} (r_2 - \tilde w_N^d)  \gamma_t^d(\cdot,\dd v, \dd x^d, \dd v^d).
\end{align}
As in Step 4 of the proof of Lemma \ref{lem.diff.u.W_2} we thus observe that $r_1 - r_2 = \Br^{-1}_{\rho}[r_2 - \tilde w_N^d]$.
An application of \eqref{est:A_rho} to \eqref{r_1-r_2} yields
\begin{align}  \label{I_2^1}
    I_2^1\lesssim \|r_2 - \tilde w_N^d\|_{L^2_{\varrho}}^2 + \sup_{x\in \R^3}\left\|r_2- \tilde w_N^d\right\|_{L^2(B_1(x))}^2.
\end{align}
Combining \eqref{I_2^2}, \eqref{I_2^3} and \eqref{I_2^1} yields \eqref{I_2}.

\medskip

\noindent\textbf{Step 5:} \emph{Proof that for $d = N^{-4/9}$
\begin{align} \label{St.rho-varrho}
         \|r_2 - \tilde w_N^d\|^2_{L^2_\varrho} + \sup_{x\in \R^3}\left\|r_2- \tilde w_N^d\right\|_{L^2(B_1(x))}^2 \lesssim 
         H+N^{-8/9}. 
    \end{align}
}
For the proof, it suffices to show that 
\begin{align}
         \|r_2 - \tilde w_N^d\|^2_{L^2_\sigma} \lesssim 
         H+N^{-8/9}. 
    \end{align}
    for all $\sigma \in \mathcal P(\R^3) \cap L^{1 + K/3}(\R^3)$ (with the implicit constant also depending on $\|\sigma\|_{1+K/3}$).
The proof  combines ideas from the estimate in Step 3 of the proof of Lemma \ref{lem.diff.u.W_2} and 
Step 6 of the proof of \cite[Theorem 1.1]{HoferSchubert23}.

\medskip

\noindent\textbf{Substep 5.1:} \emph{Splitting.} We observe that
\begin{align}
    (r_2 - \tilde w_N^d)(x) &= \int \Phi(x-z^d)(\tilde w_N^d(z^d) - v^d) - \Phi(x-z)(\tilde w_N^d(z) - v^d) \dd \gamma^d_t(z,v,z^d,v^d)  \\
    &=: \psi_1(x) + \psi_2(x) + \psi_3(x), 
\end{align}
where
\begin{align}
    \psi_1(x) &\coloneqq  \int (\Phi(x-z^d)(w_N^d(z^d) - v^d)  - \Phi(x-z)(w_N^d(z) - v^d))  \1_{\{|x - z^d| > 2d\}} \dd \gamma_t^d(z,v,z^d,v^d), \\
     \psi_2(x) &\coloneqq  \int (\Phi(x-z^d)(w_N^d(z^d) - v^d)  - \Phi(x-z)(w_N^d(z) - v^d))  \1_{\{|x - z^d| \leq 2d\}} \dd \gamma_t^d(z,v,z^d,v^d), \\
    \psi_3(x) &\coloneqq  \int \Phi(x-z^d)(\tilde w_N^d(z^d) - w_N^d(z^d)) - \Phi(x-z)(\tilde w_N^d(z) - w_N^d(z))  \dd \gamma_t^d(z,v,z^d,v^d).
\end{align}
The reason why we replace $\tilde w_N^d$ by $w_N^d$ is that  $w_N^d$ is controlled in $W^{1,\infty}$
due to Lemma \ref{lem:F.w_N} and the assumption $d=N^{-4/9}$, while such a control for $\tilde w_N^d$ does not hold true.

\medskip

\noindent\textbf{Substep 5.2:} \emph{Estimate of $\psi_1$.}

Analogously to Step 3 of the proof of Lemma \ref{lem.diff.u.W_2}, we have
\begin{align} \label{eq:psi1}
\begin{aligned}
    |\psi_1(x)| 
    & \lesssim (1+\|w_N^d \|_{W^{1,\infty}}) \left( \int |z^d-z|^2 \left(1 + \frac{1}{|x-z^d|^{\frac 9 4}} + \frac{1}{|x-z|^{\frac 9 4}}  \right)  \dd  \gamma^d(z,v,z^d,v^d) \right)^{\frac 1 2}  \\
    & \times \left( \int 1 + \frac{1}{|x-z^d|^{\frac 9 4}} \1_{\{|x - z^d| > 2d\}} + \frac{1}{|x-z|^{\frac 9 4}}   \dd  \gamma^d(z,v,z^d,v^d) \right)^{\frac 7 {18}} \\
    & \times \left( 1 + \int  |v^d|^9  \dd  \gamma^d(z,v,z^d,v^d)\right)^{\frac 1 {9}}.
    \end{aligned}
\end{align}
The last factor on the right-hand side satisfies
\begin{align}
   \int  |v^d|^9  \dd  \gamma^d(z,v,z^d,v^d) = \int |v^d|^9 \dd f^d_N = \frac 1 N |V|_9^9 \lesssim 1
\end{align}
due to \eqref{moments.V.F}.
For the second factor on the right-hand side of \eqref{eq:psi1}, due to  $d \leq \dmin/4$, we observe that 
\begin{align} 
    \sup_{x \in \R^3} \int \frac{1}{|x-z^d|^{\frac 9 4}} \1_{\{|x - z_d| > 2d\}} \dd \gamma^d(z,v,z^d,v^d) &\lesssim \frac 1 N \left(\sup_{i} \sum_{j \neq i} |X_i - X_j|^{-9/4}+d^{-9/4}\right) \\
    & \leq  \frac 1 N (S_{9/4}+d^{-9/4}) \lesssim 1,
\end{align}
where we used \eqref{eq:S_est}.
Arguing as in Step 3 of the proof of Lemma \ref{lem.diff.u.W_2} for the remaining terms, and using \eqref{w_N.Lipschitz} yields
\begin{align} \label{psi.1}
\begin{aligned}
   \|\psi_1\|^2_{L^2_\sigma} &\lesssim  \int |z^d-z|^2  \dd  \gamma_t^d(z,v,z^d,v^d)\\ 
    &\lesssim \int |z^N-z|^2  \dd  \gamma_t(z,v,z^N,v^N) + \int |z^N-z^d|^2  \dd  \gamma_t^{N,d}(z^N,v^N,z^d,v^d) \\
   & \lesssim H + \int_{\R^3} |z^d - T^d(z^d)|^2  \dd  \gamma_t^{N,d}(z^N,v^N,z^d,v^d) \\
   & \lesssim H + d^2\ls H+N^{-8/9},
   \end{aligned}
\end{align}
where we used the definitions of $\gamma^{N,d}$ and $T^d$ after \eqref{r_2} in the step to the last line.

\medskip

\noindent\textbf{Substep 5.3:} \emph{Estimate of $\psi_2$.}
Regarding $ \psi_2$, on the one hand we estimate, using $\|w_N^d\|_\infty \lesssim 1$ due to \eqref{w_N.Lipschitz} and $d = N^{-4/9}$ as well as using that the balls $B_{3d}(X_i)$ are disjoint 
\begin{align}
    & \int \left(\int |\Phi(x-z^d)| (|w_N^d(z^d)| + |v^d|) \1_{\{|x - z^d| \leq 2d\}} \dd \gamma^d(z,v,z^d,v^d)\right)^2 \dd \sigma(x) \\
     &  \qquad \qquad \lesssim \int \left( \frac 1 N \sum_i \fint_{B_d(X_i)} \frac{1 + |V_i|}{|x-y|}  \dd y  \1_{B_{3d}(X_i)}(x) \right)^2 \dd \sigma(x) \\
     &  \qquad \qquad  \lesssim \int \frac 1 {N^2}  \sum_i \frac{1+|V_i|^2}{d^2}  \1_{B_{3d}(X_i)}(x) \dd \sigma(x) 
      \\
      & \qquad \qquad \lesssim  \frac 1 {N^2}  (1 + |V|^2_\infty)  \|\sigma\|_1\\
      &  \qquad \qquad  \lesssim \frac 1{N^{\frac {16}{9}} d^2} \left(1 + N^{-2/9} |V|^2_9 \right) \\
          &  \qquad \qquad  \lesssim \frac 1{N^{16/9} d^2} = N^{-8/9}.
\end{align}
For the second and third estimate we used that for fixed $x \in \R^3$ at most one term in the sum is nonzero which allows us to interchange the square and the sum.
On the other hand,
\begin{align}
   & \int \left(\int |\Phi(x-z)| (|w_N^d(z)|+  |v^d|) \1_{\{|x - z^d| \leq 2d\}} \dd \gamma^d(z,v,z^d,v^d)\right)^2 \dd \sigma(x) \\
   &\qquad \qquad \lesssim \int \bra{\int |x-z|^{-9/4} \dd \rho(z)}^{8/9}  N^{-2/9}|V|_9^2  \left(\int \1_{\{|x-z^d|<2d\}} \dd \rho_N^d(z^d) \right)^{8/9} \dd \sigma(x)\\
   &\qquad \qquad\ls N^{-8/9}
\end{align}
Thus, 
\begin{align} \label{psi.2}
    \|\psi_2\|^2_{L^2_\sigma} \lesssim N^{-8/9}.
\end{align}

\medskip

\noindent\textbf{Substep 5.4:} \emph{Estimate of $\psi_3$.}

    We have
    \begin{align}
        |\psi_3(x) | &=  |\Phi \ast \left((\rho_N^d - \varrho)(w_N^d -\tilde w_N^d)\right)|(x) \\
        &\lesssim
        \left(\norm{\frac 1 {|x - \cdot|}}_{L^2_{\rho^d}} + \norm{\frac 1 {|x - \cdot|}}_{L^2_{\rho_N^d}}\right) \left(
        \|w_N^d -\tilde w_N^d\|_{L^2_{\rho_N^d }} + \|w_N^d -\tilde w_N^d\|_{L^2_{\varrho}}\right)
    \end{align}
    Thus, using
    \begin{align}
            \|w_N^d -\tilde w_N^d\|^2_{L^2_{\varrho}}\le \|w_N^d -\tilde w_N^d\|^2_6\|\varrho\|_{3/2}^{2/3}\ls \|\nabla(w_N^d -\tilde w_N^d)\|^2_2\|\varrho\|_{3/2}^{2/3},
        \end{align}
     and applying Fubini  and Lemma \ref{lem:A_rho[V]-w_N} yields
    \begin{align} \label{psi.3}
        \|\psi_3\|_{L^2_{\sigma}} \lesssim
        \sup_y \norm{\frac 1 {|\cdot - y|}}_{L^2_{\sigma}} \left(
        \|w_N^d -\tilde w_N^d\|_{L^2_{\rho_N^d }} + \|w_N^d -\tilde w_N^d\|_{L^2_{\varrho}}\right) \lesssim N^{-4/9}. \qquad
        \end{align}
\medskip

\noindent\textbf{Substep 5.5:} \emph{Conclusion.}

Combining \eqref{psi.1}, \eqref{psi.2} and \eqref{psi.3} yields the assertion.

\medskip

\noindent\textbf{Step 6:} \emph{Conclusion of estimate \eqref{E.dot}.} 

Collecting estimates \eqref{E.dot.0}, \eqref{I_1}, \eqref{I_2} and \eqref{St.rho-varrho} yields \eqref{E.dot}.

\medskip

\noindent\textbf{Step 7:} \emph{Estimate for the fluid velocity fields.}

Following Steps 1 and 2 of the proof of Lemma \ref{lem.diff.u.W_2}, we observe that $w := u - r_1$ (cf. \eqref{r_1}) satisfies
   \begin{align}
        -\Delta w + \nabla p = \int (v - v^N - w) \gamma_t(\cdot,\dd v, \dd x^N, \dd v^N), \quad \dv w = 0,
    \end{align}
    and hence testing with $w$ itself yields
\begin{align} \label{u.r_1}
    \|u - r_1\|^2_{L^2(B_1(x))} \lesssim  \|w\|^2_{L^6(B_1(x))}\ls  \|\nabla w\|^2_{2} \ls E.
\end{align}
Next, recalling the definition of $I_2^1, I_2^2$ from \eqref{def.I_2^2} and combining estimates \eqref{I_2^2}, \eqref{I_2^1} and \eqref{St.rho-varrho} yields 
\begin{align}\label{r_1.u_N^d}
    \|r_1 - w_N^d\|^2_{L^2(B_1(x))} \lesssim H+N^{-8/9}.
\end{align}
Finally, by \eqref{eq:u.w.L^2_loc} and $d = N^{-4/9}$, we have
\begin{align} \label{w_N^d.u_N}
    \|w_N^d - u_N\|^2_{L^2(B_1(x))} \lesssim N^{-13/9}. 
\end{align}
Combining \eqref{u.r_1}, \eqref{r_1.u_N^d} and \eqref{w_N^d.u_N} yields
\begin{align}
    \|u_N - u\|^2_{L^2(B_1(x)} \lesssim N^{-8/9} + H + E
\end{align}
Using again \eqref{eq:EH} and \eqref{lower.discretization.error}, the estimate for the fluid velocity fields in \eqref{eq:est_main} follows.
\end{proof}

%% file: Appendix.tex
\appendix

\section{Gronwall type estimates}

We list here the ODE arguments used in the article.
\begin{lem}\label{l:ODE}
    Let $a,b:[0,T]\to [0,\infty)$ be absolutely continuous and $d\ge 0$, such that there exist constants $C\in [1,\infty), c>0$ with
    \begin{align}
        \frac {\dd}{\dd t} a&\leq C b \label{eq:ddta}\\
        \frac {\dd}{\dd t} b &\leq \lambda\left(-c b +  C a+d\right).\label{eq:ddtb}
    \end{align}
    Then 
    \begin{align}
        a(t)&\le a(0)e^{Ct}+\bra{\frac dC+\frac{b(0)}{c\lambda+C}}(e^{Ct}-e^{-c\lambda t}),\label{eq:a_est}\\
        b(t)&\le b(0)e^{-c\lambda t}+ C\bra{a(0)+\frac{b(0)}{c\lambda+C}+2d}(e^{Ct}-e^{-c\lambda t}).\label{eq:b_est}
    \end{align}
\end{lem}
\begin{proof}
We introduce the notation $\bar a(t)=\sup_{0\le s\le t}a(s)$. Then \eqref{eq:ddtb} entails
\begin{align}
    b(t)&\le e^{-c\lambda t}b(0)+\lambda\int_0^t e^{-\lambda c(t-s)}(Ca(s)+d))\dd s\le e^{-c\lambda t}b(0)+(C\bar a+d)(1-e^{-\lambda ct})\label{eq:b_int}\\
    &\le e^{-\lambda c t}b(0)+C\bar a+d,\nonumber
\end{align}
We combine this with 
\begin{align}
     \frac {\dd}{\dd t} \bar a &\leq Cb,
\end{align}
which is a direct consequence of \eqref{eq:ddta}. An application of Gronwall's inequality yields \eqref{eq:a_est} and \eqref{eq:b_est} follows by inserting \eqref{eq:a_est} into \eqref{eq:b_int} and appropriate simplification.
\end{proof}

We cite the following lemma for completeness.
\begin{lem}[{\cite[Lemma 3.3]{Hofer18InertialessLimit}}]
	\label{lem:ODEEstimates}
	Let $T>0$ and $a,b:[0,T] \to \R_+$ be Lipschitz continuous. Let $\alpha \colon [0,T] \to \R_+$ be continuous and 
	$\lambda \geq 4 \max\{ 1,\|\alpha\|_{L^\infty(0,T)} \}$. 
	Let $\beta \geq 0$ be some constant and assume that on $(0,T)$
		\begin{align}
			|\dot{a}| &\leq b, \\
			\dot{b} &\leq \lambda(\alpha a - b) + \beta e^{-\lambda s}.
		\end{align}
	\begin{enumerate}[label=(\roman*)]
		
	\item	\label{it:ODEEstimates1}
		If $a(T) = 0$, then for all $s,t \in [0,T]$ with $s \leq t$
		\begin{align}
			a(t) &\leq \frac{2}{\lambda} b(t) + \frac{4}{\lambda^2} \beta e^{-\lambda t}, \label{eq:ODEEstimates1a} \\
			b(t) & \leq  \exp \left(\int_s^t -\lambda + 2 \alpha(\tau) \dd \tau \right) 
						\left( b(s) + \frac{2 \beta}{\lambda} e^{-\lambda s} \right). \label{eq:ODEEstimates1b}
		\end{align}

	\item  \label{it:ODEEstimates2}
	If $\beta = 0$ and $b(0) = 0$, then for all $t \in [0,T]$
		\[
			b(t) \leq  2 \|\alpha\|_{L^\infty(0,T)} a.
		\]
	\end{enumerate}
\end{lem}

\section{Stokes type law  proved in \texorpdfstring{\cite{HoferSchubert23}}{HS23}}

In this section, we recall a lemma from \cite{HoferSchubert23} that plays an important role in this article. 

\begin{lem}{\cite[Lemma 5.1]{HoferSchubert23}} \label{le:force.representation}
Let $r>0$, and $d \geq 4 r$.
Let $A = B_d(0) \setminus B_{d/2}(0)$ and $\tilde A = B_d(0) \setminus B_{r}(0)$.
Then, there exists a weight $\omega \colon A \to \R^{3\times3}$ with
\begin{align} \label{omega}
    \fint_A \omega \dd x = \Id, \qquad    \|\omega\|_\infty \leq C 
\end{align}
for some universal constant $C$ and such that the following holds true.
For all $(w,p) \in H^1(\tilde A) \times L^2(\tilde A)$ with  
\begin{align}\label{eq:Stokes_A}
    - \Delta w + \nabla p = 0, ~ \dv w = 0 \quad \text{in} ~ \tilde A
\end{align}
we have
\begin{align}\label{eq:br}
    -\int_{\partial B_r(0)} \sigma[w] n \dd \mathcal H^2 = 6 \pi r\left( \fint_{\partial B_r(0)} w \dd \mathcal H^2 - \fint_A \omega w \dd x\right).
\end{align}
\end{lem}